\documentclass[12pt]{article}
\usepackage{amsmath,amssymb,amsthm,url}
\usepackage{pgf,tikz}
\usetikzlibrary{arrows}
\input{xy}\xyoption{all}
\input epsf

\graphicspath{{ktori_figures/}}

\def\Z{{\mathbb Z}} \def\R{{\mathbb R}} \def\Q{{\mathbb Q}} 
\long\def\comment#1\endcomment{}
\def\inc{\mathop{\fam0 i}}
\def\pr{\mathop{\fam0 pr}}
\def\rk{\mathop{\fam0 rk}}
\def\forg{\mathop{\fam0 forg}}

\def\con{\mathop{\fam0 Con}}
\def\id{\mathop{\fam0 id}}

\def\Int{\mathop{\fam0 Int}}
\def\Cl{\mathop{\fam0 Cl}}

\def\im{\mathop{\fam0 im}}

\def\coker{\mathop{\fam0 coker}}

\def\cs{\phantom{}_{\inc}\#}

\theoremstyle{theorem}
\newtheorem{Theorem}{Theorem}[section]
\newtheorem{Lemma}[Theorem]{Lemma}
\newtheorem{Corollary}[Theorem]{Corollary}

\newtheorem{Conjecture}[Theorem]{Conjecture}

\theoremstyle{definition}
\newtheorem{Remark}[Theorem]{Remark}

\begin{document}

\comment

I am grateful for suggestions of the Referee and the Editor.
All of them are incorporated.
In particular,

* The phrase `In this paper the sign $\circ$ of the composition is often omitted'
is moved before Corollary 1.5, where it is first used.
I would agree to have it even earlier, but then there is a danger
of the reader forgetting that phrase when it is used.

* The phrase

`Methods of those papers work hardly for connected manifolds, for $2m<3n+3-d$, for high-dimensional manifolds and without the stronger dimension restriction, respectively.'

is changed to

`The methods of those papers essentially use the restrictions present there.'

I would agree to add
`i.e. that the manifolds are disjoint unions of spheres,  that $2m\ge3n+3-d$, that the manifolds are 3- or 4-dimensional, and the stronger dimension restriction, respectively.'

{\bf \uppercase{Classification of knotted tori.} A. Skopenkov}

\bigskip

Moscow Institute of Physics and Technology and Independent University of Moscow.

Homepage: \url{www.mccme.ru/~skopenko}.

\bigskip

Links and embeddings of highly-connected manifolds in codimension $>2$ were classified by Haefliger and Hirsch in 1960s.
A classification of {\it knotted tori}, i.e. embeddings $S^p\times S^q\to\R^m$ up to isotopy, is a natural next step towards classification of embeddings of {\it arbitrary} manifolds.
Many interesting examples of such embeddings and classification for particular cases were given by
Alexander (1924), Kosinski (1961), Hudson (1963), Wall (1965), Tindell (1969), Boechat and Haefliger 1970-71,  Milgram-Rees (1971), Lucas-Saeki 2002, Skopenkov (2002), see [M].
Classification of knotted tori gives some information for arbitrary manifolds and reveals new interesting relations to algebraic topology.

For a smooth manifold $N$ denote by $E^m(N)$ the set of smooth isotopy classes of smooth embeddings $N\to\R^m$.
A description of the set $E^m(S^p\times S^q)$ was known only for $p=0$, $m\ge q+3$ (Haefliger) or for $2m\ge 3p+3q+4$ \cite{Sk02} (in terms of homotopy groups of spheres and Stiefel manifolds).
For $m\ge2p+q+3$ we introduce an abelian group structure on $E^m(S^p\times S^q)$.
We prove that this group and
$$E^m(D^{p+1}\times S^q)\oplus\ker\lambda_U\oplus E^m(S^{p+q})$$
are `isomorphic up to an extension problem'.
Here $\lambda_U:E\to\pi_q(S^{m-p-q-1})$ is the linking coefficient defined on the subset
$E\subset E^m(S^q\sqcup S^{p+q})$ formed by isotopy classes of embeddings whose restriction to each component is unknotted.

This result \cite{Sk11} and its proof have corollaries which, under stronger dimension restrictions, more explicitly describe $E^m(S^p\times S^q)$ in terms of homotopy groups of spheres and Stiefel manifolds.
In the proof we use a recent exact sequence of M. Skopenkov [Sk11].

\bigskip
[M] Manifold Atlas Project, \url{http://www.map.him.uni-bonn.de/Knotted_tori}

[Sk02] A. Skopenkov,{\it On the Haefliger-Hirsch-Wu invariants for embeddings and
immersions,} Comment. Math. Helv. 77 (2002), 78--124.

[Sk11] M. Skopenkov, {\it When is the set of embeddings finite?} Intern. J. Math., to appear, arxiv: math/1106.1878.

[Sk11] A. Skopenkov, {\it Classification of knotted tori,} arxiv: math/1502.04470

\centerline{\uppercase{\bf  Classification of knotted tori}}
\bigskip
\centerline{\bf A. Skopenkov}

\bigskip
The classical Knotting Problem is as follows: {\it for an $n$-manifold $N$
and a number $m$ describe the set $E^m(N)$ of isotopy classes of embeddings $N\to\R^m$}.
Various methods for the investigation of the Knotting Problems (in higher dimensions)
were created by classical mathematicians.
For recent surveys on the Knotting Problem see [RS99, S08, HCEC];

Many interesting examples of embeddings are embeddings $S^p\times S^q\to S^m$, i.e. knotted tori [KT].
A classification of knotted tori is a natural next step (after the link theory
and the classification of embeddings of highly-connected manifolds) towards
classification of embeddings of arbitrary manifolds.
Since the general Knotting Problem is very hard, it is very interesting
to solve it for the important particular case of knotted tori.
Recent classification results for knotted tori [S06, CRS07, CRS12, S11]
give some insight or even precise information concerning arbitrary manifolds
(cf. [S07, S10, PCS]) and reveal new interesting relations to algebraic topology.

We work in the smooth category.
For a smooth manifold $N$ denote by $E^m(N)$ the set of smooth embeddings $N\to\R^m$ up to smooth isotopy.
Denote by $V_{k,l}$ the Stiefel manifold of $l$-frames in $\R^k$.
Consider [S06, S08, PCS, S11]

$\bullet$ a `connected sum' group structure on $E^m(S^q)$ for $m\ge q+3$,

$\bullet$ an `$S^p$-parametric connected sum' group structure on $E^m(S^p\times S^q)$ for $m\ge2p+q+3$  and

$\bullet$ a `connected sum of $q$-spheres together with fields of $p$ normal vectors' or
`$D^p$-parametric connected sum' group structure on $E^m(D^p\times S^q)$ for $m\ge q+3$.

\smallskip
{\bf Theorem.} {\it For $1\le p\le q$ and $m\ge2p+q+3$ groups
$$E^m(S^p\times S^q)\quad\text{and}\quad E^m(D^{p+1}\times S^q)\oplus\ker\lambda_{12}\oplus E^m(S^{p+q})
\quad{are \ adjoint}.$$
Here $\lambda_{12}:E^m_{PL}(S^q\sqcup S^{p+q})\to\pi_q(S^{m-p-q-1})$ is the linking coefficient.
The `isomorphism' from the right to the left is $\tau'\oplus\mu\oplus\#$, where $\tau'$ is the restriction map,
$\#$ is the connected sum map and $\mu$ is defined in [S11].}

\smallskip
{\bf Conjecture.} {\it For $1\le p\le q$ and $m\ge2p+q+3$}
$$E^m(S^p\times S^q)\cong E^m(D^{p+1}\times S^q)\oplus\ker\lambda_{12,0}\oplus E^m(S^{p+q}).$$

For $2m\ge 3q+2p+4$ and $2m\ge 3q+2p+3$ Theorem and Conjecture are proved in [S02] and in [S06], respectively.

Theorem (and its improvements) allows many explicit calculations.
E.g.

$\bullet$ $E^{10}(S^1\times S^5)\cong\Z_4$;

$\bullet$ $E^{11}(S^1\times S^6)$ is adjoint to $\Z_2\oplus\Z\oplus E^{11}(S^7)$, of which $E^{11}(S^7)$ is rank one infinite and splits off;

$\bullet$ $E^{12}(S^1\times S^7)$ is adjoint to $\Z^2\oplus\Z_2^3\oplus E^{12}(S^8)$, of which $E^{12}(S^8)$ is finite and splits off.

A description of $E^m(S^p\times S^q)\otimes\Q$ was obtained in [CRS07, CRS12, S11].
However, looking at the exact sequences there it is non-trivial even to guess a simple description as above.

\bigskip
{\bf References}


[CRS07] M. Cencelj, D. Repov\v s and M. Skopenkov, Homotopy type of the complement
of an immersion and classification of embeddings of tori, Uspekhi Mat. Nauk, 62:5 (2007) 165-166.
English transl: Russian Math. Surveys, 62:5 (2007). arxiv:math/0803.4285

[CRS12] M. Cencelj, D. Repov\v s and M. Skopenkov,
Classification of knotted tori in the 2-metastable dimension, Mat. Sbornik, 203:11 (2012), 1654-1681,
arXiv:math/0811.2745

[HCEC] http://www.map.him.uni-bonn.de/index.php/

High\_codimension\_embeddings:\_classification

[KT] http://www.map.him.uni-bonn.de/index.php/Knotted\_tori

[PCS] http://www.map.him.uni-bonn.de/index.php/Parametric\_connected\_sum

[RS99] D. Repov\v s and A. Skopenkov, New results on embeddings of polyhedra and
manifolds into Euclidean spaces (in Russian), Uspekhi Mat. Nauk 54:6 (1999) 61--109.
English transl.: Russ. Math. Surv. 54:6 (1999), 1149--1196.


[S02] A. Skopenkov, On the Haefliger-Hirsch-Wu invariants for embeddings and immersions,
Comment. Math. Helv. 72 (2002) 78--124.

[S06] A. Skopenkov, Classification of embeddings below the metastable dimension,
\linebreak
arxiv:math/0607422.

[S07] A. Skopenkov, A new invariant and parametric connected sum of embeddings, Fund. Math. 197 (2007) 253--269.
arxiv:math/0509621

[S08] A. Skopenkov, Embedding and knotting of manifolds in Euclidean spaces,
in: Surveys in Contemporary Mathematics, Ed. N. Young and Y. Choi, London Math. Soc. Lect. Notes
347 (2008) 248--342. arxiv:math/0604045

[S10] A.  Skopenkov, Embeddings of $k$-connected $n$-manifolds into $\R^{2n-k-1}$,
Proc. AMS, 138 (2010) 3377-3389, arxiv:math/0812.0263

[S11] M. Skopenkov, When the set of embeddings is finite? submitted, arxiv:math/1106.1878.

\endcomment

\title{Classification of knotted tori}

\author{A. Skopenkov
\footnote{Moscow Institute of Physics and Technology, and Independent University of Moscow.
 E-mail: skopenko@mccme.ru, homepage: \texttt{www.mccme.ru/\~{}skopenko}.
\newline
This work is supported in part by the Russian Foundation for Basic Research Grants 15-01-06302 and 19-01-00169, by Simons-IUM Fellowship and by the D. Zimin's Dynasty Foundation Grant.
I am grateful to P. Akhmetiev, S. Avvakumov, M. Grant, S. Melikhov, M. Skopenkov, A. Sossinsky and A. Zhubr
for useful discussions.
\newline
{\it MSC classes:} 57R40, 57R52. {\it Keywords:} embedding, isotopy, links, knotted tori, Stiefel manifolds.}
}

\date{}

\maketitle

\begin{abstract}
For a smooth manifold $N$ denote by $E^m(N)$ the set of smooth isotopy classes of smooth embeddings $N\to\R^m$.
A description of the set $E^m(S^p\times S^q)$ was known only for $p=q=0$ or for $p=0$, $m\ne q+2$ or for
$2m\ge 2(p+q)+\max\{p,q\}+4$.
(The description was given in terms of homotopy groups of spheres and of Stiefel manifolds.)
For $m\ge2p+q+3$ we introduce an abelian group structure on $E^m(S^p\times S^q)$ and describe this group
`up to an extension problem'.
This result has corollaries which, under stronger dimension restrictions, more explicitly describe $E^m(S^p\times S^q)$.
The proof is based on relations between sets $E^m(N)$  for different $N$ and $m$, in particular,
on a recent exact sequence of M. Skopenkov.
\end{abstract}

\tableofcontents

\section{Introduction and main results}\label{s:intr}

\subsection{Some general motivations}\label{0genr}

This paper is on the classical Knotting Problem: {\it for an $n$-manifold $N$
and a number $m$, classify isotopy classes of embeddings $N\to\R^m$}.
For recent surveys see \cite{Sk06, Sk16c}; whenever possible I refer to these surveys not to original papers.

I consider {\it smooth} manifolds, embeddings and isotopies.
By a classification I mean {\it a readily calculable} classification.\footnote{For a discussion of the adjectives `smooth', `readily calculable', and of embeddings into $\R^m$ vs into $S^m$ see \cite[Remark 2.20]{CS16}, \cite[Remarks 1.1 and 1.2]{Sk16c}.}
Main results are stated in \S\ref{0stat}, \S\ref{0meth} independently of \S\ref{0genr}.

Many interesting examples of embeddings are embeddings $S^p\times S^q\to\R^m$, i.e. {\it knotted tori}.
See references in \cite{Sk16k}.
Since the general Knotting Problem is very hard \cite{Sk16c}, it is very interesting to solve it for the important particular case of knotted tori.
Classification of knotted tori is a natural next step after the Haefliger link theory \cite{Ha66C} and the classification of embeddings of highly-connected manifolds \cite[\S2]{Sk06}, \cite{Sk16e}.
Such a step gives some insight or even precise information concerning embeddings of {\it arbitrary} manifolds   \cite{Sk05i, Sk08, Sk14}, and reveals new interesting relations to algebraic topology.

The Knotting Problem is more accessible for $2m\ge3n+4$, when there are some classical complete readily calculable classifications of embeddings \cite[\S2, \S3]{Sk06}, \cite{Sk16c}.
Cf. (S) of \S\ref{0metca}.

The Knotting Problem is harder for $2m<3n+4$:
if $N$ is a closed manifold that is not a disjoint union of homology spheres,
then until recently no complete readily calculable isotopy classification was known.
Even no description of embeddings `modulo knots' was known, except for $m=6k+1$, $n=4k$ and $(2k-1)$-connected $N$ \cite{BH70, Bo71}.
This is in spite of the existence of many interesting approaches including methods of Haefliger-Wu, Browder-Wall
and Goodwillie-Weiss \cite[\S5]{Sk06}, \cite{Wa70, GW99, CRS04}.


Classification results for $2m<3n+4$ concern links \cite{Ha66C, CFS, Av14, Av17}, embeddings of $d$-connected $n$-manifolds for $2m\ge3n+3-d$ \cite{Sk97, Sk02}, embeddings of 3- and 4-dimensional manifolds \cite{Sk06', Sk05, CS08, CS16, CS16o}, and rational classification of embeddings $S^p\times S^q\to\R^m$ under stronger dimension restriction than $m\ge2p+q+3$ \cite{CRS07, CRS} (see footnote \ref{f:used}).
The methods of those papers essentially use the restrictions present there.

The new ideas allowing to go beyond the above results follow \cite{Sk11} and unpublished work \cite{Sk06''}.
One idea is to find {\it relations} between different sets of (isotopy classes of) embeddings,
invariants of embeddings and geometric constructions of embeddings.
Group structures on sets of embeddings are constructed (following \cite{Sk06''}).\footnote{\label{f:used} These group structures are already used in \cite{CRS07, CRS, Sk11}.}
Then such relations are formulated in terms of exact sequences.
The most non-trivial exact sequence is relation of knotted tori to links and {\it knotted strips}
$D^p\times S^q\to S^m$, i.e. the $\nu\sigma(i\zeta\lambda')$-sequence from the proof of Theorem \ref{t:main} in \S\ref{0modulo}.
This is the main theoretical result \cite[Theorem 1.6]{Sk11} of \cite{Sk11}, which non-trivally extends
\cite[Restriction Lemma 5.2]{Sk06''} and Lemma \ref{l:exact}.a
(see footnote \ref{f:used}).

This theoretical result yielded rational classification (Corollary \ref{t:cornum}.a \cite[Corollary 1.7]{Sk11}).
Still, it was expected that embeddings $S^p\times S^q\to\R^m$ are hard to classify for $m\ge2p+q+3>q+3$.
Such a classification is the main result of this paper.
The main idea of this paper is, in some sense, {\it a reduction} of classification of knotted tori to classification of links and knotted strips (rather than {\it a relation} as in \cite{Sk11}).
This is obtained by discovering new relations between different sets of embeddings, and, more importantly,
{\it connections} between such relations, formulated in terms of diagrams involving the exact sequences,
see \S\ref{0modulo}.
These ideas are hopefully interesting in themselves.


\subsection{Statements of main results}\label{0stat}

For a manifold $N$ let $E^m(N)$ be the set of isotopy classes of embeddings $N\to S^m$.

The `connected sum' abelian group structure on $E^m(S^q)$ and on $E^m(S^q\sqcup S^n)$
was defined for $m-3\ge q,n$ by Haefliger \cite{Ha66A, Ha66C}.
Abelian group structures on $E^m(D^p\times S^q)$ for $m\ge q+3$ and on $E^m(S^p\times S^q)$ for $m\ge2p+q+3$ are defined analogously to the well-known case $p=0$.
The sum operation on $E^m(D^p\times S^q)$ is `connected sum of $q$-spheres together with normal
$p$-framings' or `$D^p$-parametric connected sum'.
The sum operation on $E^m(S^p\times S^q)$ is `$S^p$-parametric connected sum', cf. \cite{Sk05i, Sk08, MAP}, \cite[Theorem 8]{Sk14}.
See accurate definitions in \S\ref{0stand}; cf. Remarks \ref{r:prwo} and \ref{r:dire}.

Our main results describe the group $E^m(S^p\times S^q)$ up to an extension problem.


\smallskip
{\bf Definitions of $[\cdot]$, the `embedded connected sum' or `local knotting' action
$$\#:E^m(N)\times E^m(S^n)\to E^m(N),$$ and of $E_\#^m(N),q_\#$.}
By $[\cdot]$ we denote the isotopy class of an embedding or the homotopy class of a map.

Assume that $m\ge n+2$ and $N$ is a closed connected oriented $n$-manifold.
Represent elements of $E^m(N)$ and of $E^m(S^n)$ by embeddings $f:N\to S^m$ and $g:S^n\to S^m$ whose images are contained in disjoint balls.
Join the images of $f,g$ by an arc whose interior misses the images.
Let $[f]\#[g]$ be the isotopy class of the {\it embedded connected sum} of $f$ and $g$ along this arc,
cf. \cite[Theorem 1.7]{Ha66A}, \cite[Theorem 2.4]{Ha66C}, \cite[\S1]{Av14}.

For $N=S^q\sqcup S^n$ this construction is made for an arc joining $f(S^n)$ to $g(S^n)$.

For $m\ge n+2$ the operation $\#$ is well-defined, i.e. does not depend on the choice of the arc and of $f,g$ inside their isotopy classes.\footnote{This is proved analogously to the case $X=D^0_+$ of the Standardization Lemma \ref{l:stand}.b below, because the construction of $\#$ has an analogue for isotopy, cf. \S\ref{0progro}.}
Clearly, $\#$ is an action.

Let $E_\#^m(N)$ be the quotient set of $E^m(N)$ by this action and $q_{\#}:E^m(N)\to E^m_\#(N)$ the quotient map.
A group structure on $E^m_\#(S^p\times S^q)$ is well-defined by $q_{\#}f+q_{\#}f':=q_{\#}(f+f')$,
$f,f'\in E^m(S^p\times S^q)$,  because $(f\#g)+f'=f+(f'\#g)=(f+f’)\#g$ by definition of `$+$' in \S\ref{0stand}.

\smallskip
The following result reduces description of $E^m(S^p\times S^q)$ to description of $E^m(S^{p+q})$ and of
$E^m_\#(S^p\times S^q)$, cf. \cite{Sc71}, \cite[end of \S1]{CS08}.

\begin{Lemma}[Smoothing; proved in \S\ref{0prosmho}]\label{t:smo}
For $m\ge2p+q+3$ we have $E^m(S^p\times S^q)\cong E^m_\#(S^p\times S^q)\oplus E^m(S^{p+q})$ .
\end{Lemma}

The isomorphism of Lemma \ref{t:smo} is $q_{\#}\oplus\overline{\sigma}$, where $\overline{\sigma}$ is `surgery of $S^p\times*$' defined in \S\ref{0prosmho}.
It has the property $(q_{\#}\oplus\overline{\sigma})(f\#g)=q_{\#}(f)\oplus(\overline{\sigma}(f)+g)$ for each $f\in E^m(S^p\times S^q)$, $g\in E^m(S^{p+q})$.

Denote by $V_{s,t}$ the Stiefel manifold of $t$-frames in $\R^s$.
Identify $V_{s,1}$ with $S^{s-1}$.

Known results easily imply (see Corollary \ref{t:corlam}.a) that
$$E^m_\#(S^p\times S^q)\cong \pi_q(V_{m-q,p+1})\quad\text{for}\quad 2m\ge2p+3q+4.$$
Our main result generalizes this for $m\ge 2p+q+3$.

For $m\ge n+3$ denote  by

$\bullet$ $\lambda=\lambda^m_{q,n}:E^m(S^q\sqcup S^n)\to\pi_q(S^{m-n-1})$ the linking coefficient that is
the homotopy class of the first component in the complement to the second component, see accurate definition in \cite{Sk16h}, \cite[\S3]{Sk06}.


$\bullet$ $E^m_U(S^q\sqcup S^n) \subset E^m(S^q\sqcup S^n)$ the subset formed by the isotopy classes of embeddings
whose restriction to {\it each} component is unknotted.

$\bullet$ $K^m_{q,n}:=\ker\lambda\cap E^m_U(S^q\sqcup S^n)$; see geometric description in \cite[\S3, Definition of $\widehat{DM}^m_{p,q}$]{Sk09}.

\begin{Theorem}[proved in \S\ref{0modulo}]\label{t:main}
For $m\ge2p+q+3$ the group $E^m(D^{p+1}\times S^q)$ has a subgroup $X=X^m_{p,q}$ such that $E^m_\#(S^p\times S^q)$ has a subgroup isomorphic to $X\oplus K^m_{q,p+q}$ whose quotient is isomorphic to $E^m(D^{p+1}\times S^q)/X$.

Moreover, there are maps forming the following commutative diagram, in which the horizontal sequence is exact:
$$\xymatrix{
& X \ar[r]^{\subset}  & E^m(D^{p+1}\times S^q) \ar[dr]^{q_X} \ar[d]_{q_\#r} &         \\
0 \ar[r] \ar[ur] & X\oplus K^m_{q,p+q} \ar[r]_{\overline\mu\oplus \sigma_\#} & E^m_\#(S^p\times S^q) \ar[r]_{\overline\nu} & E^m(D^{p+1}\times S^q)/X \ar[r] & 0
 }.$$
The subgroup $X^m_{p,q}$ is finite unless $q=4k-1$ and $m=6k+p$ for some $k\ge p/2+1$, and $X^{6k+p}_{p,4k-1}$ is the sum of $\Z$ and a finite group for such $k$.
\end{Theorem}


The subgroup $X^m_{p,q}$ is the kernel of the restriction map to $E^m(D^p\times S^q)$.
The maps $\sigma_\#$, $\overline\mu$ and $\overline\nu$ are defined in \S\ref{0modulo}, $r$ is the restriction map to $E^m(S^p\times S^q)$, and $q_X$ is the quotient map.
For relation between $\pi_q(V_{m-q,p+1})$ and $E^m(D^{p+1}\times S^q)$ see Theorem \ref{l:xitaurho}.

\begin{Conjecture} \label{t:conj}
For
$m\ge2p+q+3$
$$E^m(S^p\times S^q)\cong E^m(D^{p+1}\times S^q)\oplus K^m_{q,p+q}\oplus E^m(S^{p+q}).$$
\end{Conjecture}

This is equivalent to
$E^m_\#(S^p\times S^q)\cong E^m(D^{p+1}\times S^q)\oplus K^m_{q,p+q}$ by the Smoothing Lemma \ref{t:smo}.
For more discussion see Remark \ref{r:conj}.

Known cases of Theorem \ref{t:main}, the Smoothing Lemma \ref{t:smo} (and of Conjecture \ref{t:conj}) are listed in Remark \ref{r:kno}.a.
In particular, these are new results only for
$$1\le p<q\quad\text{and}\quad 2m\le 3q+2p+3.$$
Analogous remark holds for the
corollaries from \S\ref{0meth}
(under stronger dimension restrictions they describe $E^m(S^p\times S^q)$ more explicitly).


{\it Plan of the paper} should be clear from the contents of the paper and the following diagram.
Subsections are independent on each and other (except maybe for a few references which could be ignored)
unless joined by a sequence of arrows.
$$\xymatrix{ & {\ref{0stat}} \ar[r] \ar[d]  & {\ref{0stand}} \ar[d] \ar[r] \ar@(r,dl)[rr] \ar@(dr,dl)[rrr]
& {\ref{0prosta}} & {\ref{0progro}} & {\ref{0prosmho}}  \\
{\ref{0metca}} \ar@(dr,dl)[rrr] & {\ref{0meth}} \ar[l] \ar[r] & {\ref{0modulo}} \ar[r] & {\ref{0newpro}} \ar[r]
& {\ref{0prolex}}
}$$

\subsection{Corollaries}\label{0meth}

Denote by $TG$ the torsion subgroup of an abelian group $G$.

\begin{Corollary}\label{t:cornum}
Assume that $m\ge 2p+q+3$.

$$\text{(a) \cite[Corollary 1.7]{Sk11}} \quad E^m(S^p\times S^q)\otimes\Q\cong[\pi_q(V_{m-q,p+1})\oplus E^m(S^q)\oplus K^m_{q,p+q}\oplus E^m(S^{p+q})]\otimes\Q.$$
$$(b)\quad |E^m(S^p\times S^q)|=|E^m(D^{p+1}\times S^q)|\cdot|K^m_{q,p+q}|\cdot|E^m(S^{p+q})|$$
(more precisely, whenever one part is finite, the other is finite and they are equal).
$$(c)\quad |TE^m(S^p\times S^q)|=|TE^m(D^{p+1}\times S^q)|\cdot|TK^m_{q,p+q}|\cdot|TE^m(S^{p+q})|,$$
unless $m=6k+p$ and $q=4k-1$ for some $k$.

(d) For the diagram of Theorem \ref{t:main} any $\Z^s$-direct summand of $K^m_{q,p+q}$ is mapped under $\sigma_\#$ to a $\Z^s$-direct summand in $E^m_\#(S^p\times S^q)$.

(e) For the diagram of Theorem \ref{t:main} any $\Z^s$-direct summand of $E^m(D^{p+1}\times S^q)/X$ is the image of $\Z^s$-direct summands $\Delta\subset E^m(D^{p+1}\times S^q)$ and $\Sigma\subset E^m_\#(S^p\times S^q)$
such that $r\Delta=\Sigma$.
\end{Corollary}

Parts (a,b,c) are simplified versions of Conjecture \ref{t:conj}.
Part (a) follows by Theorem \ref{t:main} and the isomorphism (DF) of \S\ref{0metca} \cite[Lemma 2.15]{CFS}.
Parts (b) and (e) follow by Theorem \ref{t:main} in a standard way.
Part (c) follows from parts (d,e) and Theorem \ref{t:main}.
Part (d) is proved in \S\ref{0metca}.

\smallskip
{\bf Comment.}
(a) The right-hand side groups of Corollary \ref{t:cornum}.a are known by Theorem \ref{t:spli} below and
\cite[Theorems 1.1, 1.9 and Lemma 1.12]{CFS}.
Thus part (a) allows calculation of $\rk E^m(S^p\times S^q)$.
This is known \cite[Corollary 1.7]{Sk11}, so Corollary \ref{t:cornum}.a is not a new result.
However, our deduction of Corollary \ref{t:cornum}.a is interesting because it uses Theorem \ref{t:main} instead of  information on the groups involved in \cite[\S4, proof of Corollary 1.7]{Sk11}; in this sense our deduction
explains why the isomorphism holds.

(b) The following diagram shows that excluding the case `$m=6k+p$ and $q=4k-1$ for some $k$' is essential for our proof of Corollary \ref{t:cornum}.c.
$$\xymatrix{  & & \Z \ar[d]^{i_2\rho_2} \ar[dr]^{\rho_2}   \\
0 \ar[r] & \Z  \ar[ur]^2 \ar[r]_{1\oplus0} & \Z\oplus\Z_2 \ar[r]_{0\oplus1} & \Z_2 \ar[r] &0
}$$

{\bf Definition of $\Z_{(s)}$ and the maps $\pr_k$,}
$$\pi_q(V_{m-q,p+1}) \overset{\tau=\tau^m_{p,q}}\to E^m(D^{p+1}\times S^q)
\overset{r=r^m_{p,q}}\to E^m(S^p\times S^q).$$
Denote by $\Z_{(s)}$ the group $\Z$ for $s$ even and $\Z_2$ for $s$ odd.
This notation should not be confused with the notation for localization.

Denote by $\pr_k$ the projection of a Cartesian product onto the $k$-th factor.

Denote by $r$ the restriction-induced map.

Represent an element of $\pi_q(V_{m-q,p+1})$ by a smooth map $x:S^q\to V_{m-q,p+1}$.
By the exponential law this map can be considered as a map $x:\R^{p+1}\times S^q\to \R^{m-q}$.
The latter map can be normalized to give a map $\widehat x:D^{p+1}\times S^q\to D^{m-q}$.
Let $\tau[x]$ be the isotopy class of the composition
$D^{p+1}\times S^q\overset{\widehat x\times\pr_2}\to D^{m-q}\times S^q\overset{\inc}\to S^m$, where
$\inc$ is the standard embedding (see accurate definition in \S\ref{0stand}) \cite{Sk16k}, \cite[\S6]{Sk06}.
Clearly, $\tau$ is well-defined and is a homomorphism.

\smallskip
In this paper the sign $\circ$ of the composition is often omitted.

\begin{Corollary}\label{t:corlam} Assume that $m\ge 2p+q+3$.

(a) If
$2m\ge2p+3q+4$,
then
$q_{\#}r\tau:\pi_q(V_{m-q,p+1})\to E^m_\#(S^p\times S^q)$ is an isomorphism.

(b) If $2m\ge p+3q+4$, then $E^m_\#(S^p\times S^q)$ and $\pi_q(V_{m-q,p+1})$ have isomorphic subgroups with isomorphic quotients.

(b') If $1\le p<k$,
then $E^{6k-p}_\#(S^p\times S^{4k-p-1})\cong\Z\oplus G_{k,p}$ for a certain
group $G_{k,p}$ such that $G_{k,p}$ and $\pi_{4k-p-1}(V_{2k+1,p+1})$ have isomorphic subgroups with isomorphic quotients.


(c) If $2m\ge 3q+4$, then $E^m_\#(S^p\times S^q)$ has a subgroup isomorphic to $\pi_{p+2q+2-m}(V_{M+m-q-1,M})$,
whose quotient and $\pi_q(V_{m-q,p+1})$ have isomorphic subgroups with isomorphic quotients.

(d) If $2m=3q+3$, then $E^m_\#(S^p\times S^q)$ has a subgroup isomorphic to $\pi_{p+2q+2-m}(V_{M+m-q-1,M})$,
whose quotient has a subgroup isomorphic to $\Z_{(m-q-1)}$, whose quotient and $\pi_q(V_{m-q,p+1})$
have isomorphic subgroups with isomorphic quotients.
\end{Corollary}

Corollary \ref{t:corlam}.a is deduced from earlier known results, and Corollaries \ref{t:corlam}.bb'cd are deduced from our results, at the end of \S\ref{0metca}.
See a direct proof of Corollary \ref{t:corlam}.a in \S\ref{0newpro}.

\smallskip
{\bf Comment.}
The case $2m=2p+3q+3$ is considered in \cite{Sk06''}.
Statements \cite[Main Theorem 1.4.AD,PL]{Sk06''} in the first two arXiv versions are false;
\cite[Main Theorem 1.3]{Sk06''} and \cite[Theorem 3.11]{Sk06} are correct.
The mistakes are corrected in the present paper except in the case $(m,p,q)=(7,1,3)$ for which see \cite{CS16}.
The mistakes were in the relation $\tau_p(w_{l,p})=2\omega_p$
of \cite[the Relation Theorem 2.7, the Almost Smoothing Theorem 2.3]{Sk06''}.

I conjecture that the groups of Corollary \ref{t:corlam}.b are in fact isomorphic.
This is a particular case of Conjecture \ref{t:conj}.
This case could hopefully be proved using ideas of \cite{Sk06''}.



\smallskip
Recall that our results are new only for $1\le p<q$ and $2m\le 3q+2p+3$ (by Remark \ref{r:kno}.a).
The smallest $m$ for which there are $p,q$ such that these inequalities hold and $m\ge 2p+q+3$
are $m=10,11,12$.
Then $p=1$ and $q=m-5$.
Hence by the Smoothing Lemma \ref{t:smo}, Theorem \ref{t:main}, Corollaries \ref{t:corlam}.b,b',c,d
and \cite{Pa56, Ha66A}

$\bullet$ $|E^{10}(S^1\times S^5)|=|E^{10}_\#(S^1\times S^5)|=4$. Cf. \cite[Example 1.4]{Sk11}.

$\bullet$ $E^{11}_\#(S^1\times S^6)\cong \Z_2\oplus\Z$ and
$E^{11}(S^1\times S^6)\cong\Z_2\oplus\Z\oplus E^{11}(S^7)$, of which $E^{11}(S^7)$ is rank one infinite.

$\bullet$ $E^{12}_\#(S^1\times S^7)\cong \Z^2\oplus G$, where $|G|$ is a divisor of 8,
and $E^{12}(S^1\times S^7)\cong \Z^2\oplus G\oplus E^{12}(S^8)$, of which $E^{12}(S^8)$ is finite.

\subsection{Calculations, proofs of corollaries, remarks}\label{0metca}

{\it The group $\pi_s(V_{m,n})$ is calculated} for many cases, see e.g. \cite{Pa56}, \cite[Lemma 1.12]{CFS}.

{\bf (V)} $\pi_s(V_{m,n})=0$ for $s>m-n$.

{\bf (V')} $\pi_{m-n}(V_{m,n})\cong \Z_{(m-n)}$ for $n>1$.

{\bf (VF)} $\pi_s(V_{m,n})$ is finite if and only if either $s=m-n$ is even, or $s=m-1$ is odd, or $4| s+1\ne m$ and $\frac s2+1<m<n+\frac s2+1$.

{\it The group $E^m(S^n)$ is calculated} for some cases when $m\ge n+3$ \cite{Ha66A, Mi72}.
In particular,

{\bf (S)} $E^m(S^n)=0$ for $2m\ge 3n+4$.

{\bf (S')} $E^m(S^n)\cong \Z_{(m-n-1)}$ for $2m=3n+3$.

{\bf (SF)} $E^m(S^n)$ is finite if and only if $n\equiv3\mod4$
and $2m<3n+4$ \cite[Corollary 6.7]{Ha66A}.

\begin{Theorem}\label{t:spli}
For $m-3\ge q,n$  we have $E^m_U(S^q\sqcup S^n)\cong \pi_q(S^{m-n-1})\oplus K_{q,n}$.
\cite[Theorem 2.4 and the text before Corollary 10.3]{Ha66C}
\end{Theorem}

{\it The group $K^m_{q,p+q}$ (or, equivalently,
$E^m_U(S^q\sqcup S^{p+q})$) is calculated}
in terms of homotopy groups of spheres and Whitehead products \cite{Ha66C, Sk09}, \cite[Theorem 1.9]{CFS}.
In particular,

{\bf (L)} $K^m_{q,p+q}=0$ for $2m\ge3q+p+4$;

{\bf (L')} $K^m_{q,p+q}\cong\pi_{p+2q+2-m}(V_{M+m-q-1,M})$ for $m\ge\frac{2p+4q}3+2$ and $M$ large.

This holds by the Haefliger Theorems \cite[Theorems 3.1 and 3.6]{Sk06}.
Also (L) follows by (L').
The isomorphism of (L') from the left to the right is defined in \cite{Ha66A}.

\smallskip
{\it The group $E^m(D^{p+1}\times S^q)$ can be calculated} using Theorem \ref{l:xitaurho} below.
E.g. by Theorem \ref{l:xitaurho}, (S), (S') and since for $2m\ge 3q+2$ the normal bundle of any embedding
$S^q\to\R^m$ is trivial \cite{Ke59}, we have the following.

{\bf (D)} $\tau:\pi_q(V_{m-q,p+1})\to E^m(D^{p+1}\times S^q)$ is an isomorphism for $2m\ge 3q+4$.

{\bf (D')} $E^m(D^{p+1}\times S^q)$ has a subgroup $\Z_{(m-q-1)}$ whose quotient is $\pi_q(V_{m-q,p+1})$ for $2m=3q+3$.

{\bf (DF)} $E^m(D^{p+1}\times S^q)\otimes\Q\cong[\pi_q(V_{m-q,p+1})\oplus E^m(S^q)]\otimes\Q$
\cite[Lemma 2.15]{CFS}.

A {\bf $p$-framing} in a vector bundle is a system of $p$ ordered orthogonal normal unit vector fields on the zero section of the bundle.

\begin{Theorem}\label{l:xitaurho}
For $m\ge q+3$ the following sequence is exact:
$$\ldots\to E^{m+1}(S^{q+1})\overset\xi\to \pi_q(V_{m-q,p+1})\overset\tau\to E^m(D^{p+1}\times S^q)
\overset\rho\to E^m(S^q)\to\ldots$$
Here $\rho$ is the restriction-unduced map and $\xi[f]$ is the obstruction to the existence of a normal $(p+1)$-framing of an embedding $f:S^{q+1}\to S^{m+1}$.
\end{Theorem}

{\bf Accurate definition of the map $\xi$.}
We follow \cite[Sketch of proof of Theorem 2.5 in p. 7]{Sk11}, cf. \cite[Theorem 2.14]{CFS}, \cite[Corollary 5.9]{Ha66A}.
Take an embedding $f:S^{q+1}\to S^{m+1}$.
Take a normal $(m-q)$-framing of the image $f(D^{q+1}_-)$ of the lower hemisphere and a normal $(p+1)$-framing of the
image $f(D^{q+1}_+)$ of the upper hemisphere.
These framings are unique up to homotopy.
Thus the hemispheres of $f(S^{q+1})$ are equipped with a $(p+1)$-framing and an $(m-q)$-framing.
Using the $(m-q)$-framing identify each fiber of the normal bundle to $f(D^{q+1}_-)$ with the space $\R^{m-q}$.
Define a map $S^q\to V_{m-q,p+1}$ by mapping point $x\in S^q$ to the $(p+1)$-framing at the point $f(x)$.
Let $\xi[f]$ be the homotopy class of this map.

\begin{proof}[Proof of Corollary \ref{t:cornum}.d]
If $m=6k+p$ and $q=4k-1$ for some $k$, then by (L) $K^m_{q,p+q}=0$, hence the corollary is trivial.
So assume that there are no $k$ such that $m=6k+p$ and $q=4k-1$.

Then by Theorem \ref{t:main} $X$ is finite.

Denote by $E$ and $E_\#$ the quotients of $E^m(D^{p+1}\times S^q)$ and of $E^m_\#(S^p\times S^q)$ by the maximal summands $\Delta,\Sigma$ of Corollary \ref{t:cornum}.e.
Then $E/X$ is well-defined and is finite.
Hence $E$ is finite.

By Theorem \ref{t:main} we have the following commutative diagram, in which the horizontal sequence is exact:
$$\xymatrix{
& X \ar[r]^{\subset}  & E \ar[dr]^{q_X} \ar[d]_{q_\#r} &         \\
0 \ar[r] \ar[ur] & X\oplus K^m_{q,p+q} \ar[r]_\varphi & E_\# \ar[r]_{\overline\nu} & E/X \ar[r] & 0
 }.$$
Here we denote by $q_\#r,q_X,\varphi,\overline\nu$ the maps corresponding to $q_\#r,q_X,\overline\mu\oplus \sigma_\#,\overline\nu$.

Since $E$ is finite, we have $q_\#rE\subset TE_\#$, so $\overline\nu|_{TE_\#}$ is surjective.

Denote by $F$ the maximal free direct summand of $K^m_{q,p+q}$.
The corollary follows because in the next paragraph we prove that if $x\in F$ and $\varphi x\ne0$ is divisible by an integer $n$, then $x$ is divisible by $n$ in $F$.

Take $y\in E_\#$ such that $\varphi x=ny\ne0$.
Since  $\overline\nu|_{TE_\#}$ is surjective, there is $z\in TE_\#$ such that $\overline\nu z=\overline\nu y$.
By exactness $y-z=\varphi t$ for some $t\in X\oplus K^m_{q,p+q}$.
Then $t=t_F+t_T$ for some $t_F\in F$ and a finite order element $t_T$.
Hence $y-\varphi t_F=z+\varphi t_T\in TE_\#$.
So $TE_\#\ni n(y-\varphi t_F)=\varphi(x-nt_F)$.
Therefore $n_1\varphi(x-nt_F)=0$ for some integer $n_1>0$.
Since $\varphi$ is injective and $x,t_F\in F$, we have $x=nt_F$.
\end{proof}

\begin{proof}[Proof of Corollaries \ref{t:corlam}.b,c,d]
These corollaries follow from Theorem \ref{t:main} and (D,L), (D,L'), (D',L'), respectively.
Here (L') is applicable because
$\max\{2p+q+3,\frac{3q+3}2\}\ge\frac{2p+4q}3+2$ (indeed, the opposite inequalities imply $4p+3<q<4p+3$).
\end{proof}

\begin{proof}[Proof of Corollary \ref{t:corlam}.b']
Denote $m=6k-p$ and $q=4k-p-1$.
Since $p<k$, we have $m\ge 2p+q+3$ and $m\ge\frac{2p+4q}3+2$.
Hence by (L') $K^m_{q,p+q}\cong\Z$ is free.
Since $2m=p+3q+3\ge3q+4$, by (D) $E^m(D^{p+1}\times S^q)\cong\pi_q(V_{m-q,p+1})$.
So the corollary follows from Corollary \ref{t:cornum}.d.
\end{proof}


\begin{proof}[Deduction of Corollary \ref{t:corlam}.a from earlier known results]
Consider the following diagram
$$\xymatrix{
\pi_q(V_{m-q,p+1}) \ar[r]^\tau &  E^m(D^{p+1}\times S^q) \ar[r]^r &  E^m(S^p\times S^q) \ar[r]^{q_{\#}} \ar@(lu,ru)[ll]_{\widehat\alpha} &
E^m_\#(S^p\times S^q)
}$$
Here $\widehat\alpha$ is a map such that $\widehat\alpha r\tau=\id$ and $\widehat\alpha(f\#g)=\widehat\alpha(f)$
for each $f\in E^m(S^p\times S^q)$ and $g\in E^m(S^{p+q})$; such a map exists by \cite[Torus Lemma 6.1]{Sk02}
($\widehat\alpha:=\rho^{-1}\sigma^{-1}\pr_1\gamma\alpha$ in the notation of that lemma).
Hence $r\tau$ is injective and $q_{\#}r\tau$ is injective.

Take any $f\in E^m(S^p\times S^q)$.
Let $f':=r\tau\widehat\alpha(f)$.
Then $\widehat\alpha(f')=\widehat\alpha(f)$.
Then by \cite[Corollary 1.6.i]{Sk02} and since the smoothing obstruction assuming values in $E^m(S^{p+q})$
is changed by $[g]\in E^m(S^{p+q})$ if $f$ is changed to $f\#g$, we obtain $q_{\#}f=q_{\#}f'=q_{\#}r\tau\widehat\alpha(f)$.
Since $q_{\#}$ is surjective, we see that $q_{\#}r\tau$ is surjective.
\end{proof}



\begin{Remark}\label{r:kno}
(a) {\it Description of known cases of the Smoothing Lemma \ref{t:smo}, Theorem \ref{t:main} and Conjecture \ref{t:conj}.}
If $2m\ge 3p+3q+4$, then

$\bullet$ the Smoothing Lemma \ref{t:smo} holds by (S);

$\bullet$ Theorem \ref{t:main} and Conjecture \ref{t:conj} are true by (S), (D), (L) and the case $2m\ge 3p+3q+4$ of Corollary \ref{t:corlam}.a \cite[Theorem 3.9]{Sk06}.


For $p\ge q$ and $m\ge 2p+q+3$ we have $2m\ge 3p+3q+4$.
Hence the Smoothing Lemma \ref{t:smo}, Theorem \ref{t:main} and Conjecture \ref{t:conj} are true.
In fact, $E^m(S^p\times S^q)=0$ by the Haefliger Unknotting Theorem \cite[Theorem 2.6.b]{Sk06}.

If $2m\ge3q+2p+4$, then

$\bullet$ the Smoothing Lemma \ref{t:smo} is \cite[Theorem 1.2.DIFF]{Sk06''}
(an alternative proof follows from \cite[Proposition 5.6]{CRS} analogously to
\cite[\S4, proof of Higher-dimensional Classification Theorem (a)]{Sk06'});

$\bullet$ Theorem \ref{t:main} and Conjecture \ref{t:conj} follow from
Corollary \ref{t:corlam}.a and (D) (which are easy corollaries of known results).

If $p=0$, then

$\bullet$ the Smoothing Lemma \ref{t:smo} holds by \cite[Theorem 2.4]{Ha66C};

$\bullet$ Theorem \ref{t:main} and Conjecture \ref{t:conj} are true.
(Indeed, by \cite[Theorem 2.4]{Ha66C} and Theorem \ref{t:spli}
$E^m_\#(S^0\times S^q)\cong \pi_q(S^{m-q-1})\oplus E^m(S^q)\oplus K^m_{q,q}$.
Since $\lambda^m_{q,q}r\tau^m_{0,q}=\id\pi_q(S^{m-q-1})$, the map $\tau^m_{0,q}$ is injective.
So by Theorem \ref{l:xitaurho} $E^m(D^1\times S^q)\cong\pi_q(S^{m-q-1})\oplus E^m(S^q)$.)

There were known weaker versions of the Smoothing Lemma \ref{t:smo} \cite[Smoothing Theorem 8.1]{Sk06''}, \cite[Proposition 5.6]{CRS}, and
rational versions of Theorem \ref{t:main} \cite{CRS, Sk11}.

In the proofs of the Smoothing Lemma \ref{t:smo} and Theorem \ref{t:main} it is not required that $1\le p<q$ and $2m\le 3q+2p+3$.
So I give new proofs of known cases $2m\ge 3q+2p+4$ (in the stronger form of Conjecture \ref{t:conj} and Corollary \ref{t:corlam}.a).
These new proofs are only interesting for $1\le p<q$.

(b) {\it Description of known results for $m\le2p+q+2$ and $p,q\ge 1$.}
If $p>q$, then $2p+q+2>2q+p+2$, so after exchange of $p,q$ the inequality $m\le2p+q+2$ remains fulfilled.
So it suffices to present description for $p\le q$.
Then $E^m(S^p\times S^q)$ is known only for $(m,p,q)=(6k,2k-1,2k)$ \cite[Theorem 2.14]{Sk06} or $m=p+q+1$, $p\ge2$ \cite{LNS};
$E^m_\#(S^p\times S^q)$ is known only for $m\ge \frac{3q}2+p+2$ \cite[Theorem 3.9]{Sk06}.
For some $m,p,q$ there are no group structures compatible with natural constructions and invariants, see Remark \ref{r:dire}.
For $(m,p,q)=(7,1,3)$ see \cite{CS16, CS16o}.
\end{Remark}

\begin{Remark}[on Conjecture \ref{t:conj}]\label{r:conj}
(a) Conjecture \ref{t:conj} is known to be true for $p=0$ or $p\ge q$ or $2m\ge 3q+2p+4$, see Remark \ref{r:kno}.a.
Conjecture \ref{t:conj} is true for $(m,p,q)\in\{(11,1,6),(17,1,10)\}$
by Corollary \ref{t:corlam}.b' because $\pi_6(V_{5,2})\cong\Z_2$ and $\pi_{10}(V_{7,2})=0$ \cite{Pa56}.

(b) A candidate for an isomorphism from right to left in Conjecture \ref{t:conj} is
$r\oplus\sigma|_{K^m_{q,p+q}}\oplus\cs$, where $r$ is the restriction-induced map, $\sigma$ is defined in \S\ref{0modulo} and $\cs$ is defined in \S\ref{0stand}.
For $p=0$, $q=4k-1$ and $m=6k$ this is an isomorphism by \cite[Corollary 2.6]{Sk24}.
Proof of Theorem \ref{t:main} does not show that $[r\oplus\sigma|_{K^m_{q,p+q}}\oplus\cs]\otimes\Q$ is an isomorphism, in spite of Corollary \ref{t:cornum}.a.

(c) Conjecture \ref{t:conj} holds by the 5-lemma under the assumption $q_\#r|_X=\pm\overline\mu$.
The assumption does not hold for $p=0$, $q=4k-1$ and $m=6k$ by \cite[Corollary 2.5]{Sk24}.
Still, Conjecture \ref{t:conj} holds for this case by \cite[Corollary 2.6]{Sk24}.
It would be interesting to know if the assumption holds under stronger restrictions, e.g. for $m\ge 3q+p+4$.
\end{Remark}


\section{Proofs}\label{0plan}

\subsection{Standardization and group structure}\label{0stand}

{\bf Definition of the inclusion $\R^q\subset\R^m$ and of $\R^m_\pm,D^m_\pm,0_k,1_k,l,T^{p,q},T^{p,q}_\pm$.}
For each $q\le m$ identify the space $\R^q$ with the subspace of $\R^m$ given by the equations $x_{q+1}=x_{q+2}=\dots=x_m=0$ \cite{Ha66A}
(note that the notation in \cite{Ha66C, Sk11} is slightly different).
Analogously identify $D^q,S^q$ with the subspaces of $D^m,S^m$.

Define $\R^m_+,\R^m_-\subset\R^m$ and $D^m_+,D^m_-\subset S^m$ by equations $x_1\ge0$ and $x_1\le0$, respectively.
Then $S^m=D^m_+\cup D^m_-$.
Note that $0\times S^{m-1}=\partial D^m_+=\partial D^m_-=D^m_+\cap D^m_-\ne S^{m-1}$.
Denote by $0_k$ the vector of $k$ zero coordinates,
$$1_k:=(1,0_k)\in S^k,\quad l:=m-p-q-1,\quad T^{p,q}:=S^p\times S^q\quad\text{and}\quad T^{p,q}_\pm:=D^p_{\pm}\times S^q.$$


Assume that $m>p+q$.
Informally, {\it the standard embedding} is the smoothing of the composition
$$D^{p+1}\times D^{q+1}\cong D^{q+1}\times D^{p+1}\cong D^{q+1}\times0_l\times \frac12 D^{p+1}\overset{\subset}\to D^{q+1}\times D^l\times  D^{p+1}\cong D^{m+1}.$$
Formally, define {\bf the standard embedding}
\footnote{The image of $T^{p,q}$ under this embedding is the boundary of a certain neighborhood of
$S^q\subset S^m$ in $f(S^{p+q+1})$, where embedding $f:S^{p+q+1}\to S^m$ is defined by
$(y,z)\mapsto(y,0_l,z)$, $y\in\R^{q+1}$.}
$$\inc=\inc\phantom{}_{m,p,q}:D^{p+1}\times D^{q+1}\to D^{m+1}\quad\text{by}\quad \inc(x,y):=(y\sqrt{2-|x|^2},0_l,x)/\sqrt2.$$
Note that $\inc(D^{p+1}\times S^q)\subset S^m$, \ $\inc(D^{p+1}\times D^q_\pm)\subset D^m_\pm$ and $\inc_{m,p,q}$ is the restriction\footnote{For a map $f:X\rightarrow Y$ and $A\subset X$, $f(A)\subset B\subset Y$, the {\it restriction}  $g:A\rightarrow B$ of $f$ is defined by  $g(x):=f(x)$.}
of $\inc_{m+1,p+1,q}$ but not of $\inc_{m+1,p,q+1}$.
Denote by the same notation `$\inc$' restrictions
of $\inc$ (it would be clear from the context, to which sets).
In particular, $\inc_{m,m-n-1,n}:D^{m-n}\times S^n\to S^m$ is defined by the natural normal framing on the inclusion $S^n\to S^m$.

Take a subset $X\subset S^p$.
A map $f:X\times S^q\to S^m$ is called {\bf standardized} if
$$f(X\times\Int D^q_+)\subset\Int D^m_+\quad\text{and}\quad f|_{X\times D^q_-}=\inc\phantom{}_{m,p,q}.$$
Cf. \cite[Remark after definition of the standard embedding in \S2]{Sk05i}.

A homotopy $F:X\times S^q\times I\to S^m\times I$
is called {\bf standardized} if
$$F(X\times\Int D^q_+\times I)\subset\Int D^m_+\times I\quad\text{and}\quad
F|_{X\times D^q_-\times I}=\inc\times\id I.$$

\begin{Lemma}[Standardization Lemma; proved in \S\ref{0prosta}]\label{l:stand}
Let $X$ denote either $D^p_+$ or $S^p$.
For $X=S^p$ assume that $m\ge2p+q+3$.

(a) Each embedding $X\times S^q\to S^m$ is isotopic to a standardized embedding.

(b) If standardized embeddings $X\times S^q\to S^m$ are isotopic, then there is a standardized isotopy between them.
\end{Lemma}

{\bf Definition of the reflections $R,R_j$.}
Let $R:\R^m\to\R^m$ be the reflection of $\R^m$ with respect to the hyperplane given by equations $x_1=x_2=0$, i.e., $R(x_1,x_2,x_3,\dots,x_m):=(-x_1,-x_2,x_3,\dots,x_m)$.
Let $R_j$ be the reflection of $\R^m$ with respect to the hyperplane $x_j=0$, i.e.,
$R_j(x_1,x_2,\dots,x_{j-1},x_j,x_{j+1},\dots,x_m):=(x_1,x_2,\dots,x_{j-1},-x_j,x_{j+1},\dots,x_m)$.

\begin{Lemma}[Group Structure Lemma; proved in \S\ref{0progro}]\label{t:grst}
Let $X$ denote either $D^p_+$ or $S^p$.
For $X=D^p_+$ assume that $m\ge q+3$, for $X=S^p$ assume that $m\ge2p+q+3$.
Then a commutative group structure on $E^m(X\times S^q)$ is well-defined  by the following construction.

Let $0:=[\inc]$.
Let $-[f]:=[\overline f]$, where $\overline f(x,y):=R_2f(x,R_2y)$.
For standardized embeddings $f,g:X\times S^q\to S^m$ let $[f]+[g]$ be the isotopy class of the embedding $s_{fg}$ defined by
$$s_{fg}(x,y):=
\begin{cases} f(x,y)& y\in D^q_+\\ R(g(x,Ry))& y\in D^q_-\end{cases}.$$
(The two formulas agree on $X\times (D^q_+\cap D^q_-)$ because $\inc(x,y)=R\inc(x,Ry)$.
Clearly, $s_{fg}$ is an embedding.)
\end{Lemma}


{\bf Define the `embedded connected sum' or `local knotting' map}
$$\cs:E^m(S^{p+q})\to E^m(T^{p,q})\quad\text{by}\quad \cs(g):=0\#g=[\inc]\#g.$$
Identify  $1\times S^q$ and $-1\times S^q$ with the {\it first} and the {\it second} component of $S^q\sqcup S^q$, respectively.
Clearly, for $m\ge 2p+q+3$ the map $\cs$ is a homomorphism.



\begin{Remark}[comparison to previous work]\label{r:prwo}
(a) The Standardization and Group Structure Lemmas \ref{l:stand} and \ref{t:grst} for $X=D^p_+$ generalize
the well-known construction of the connected sum of knots, i.e. of isotopy classes of embeddings $S^q\to S^m$.
I could not find either a proof that the connected sum is well-defined for $m=q+2=3$, or reference to such a proof,
in \cite{Ad04, BZ03, CDM, CF63, Ka87, Ma04, PS96, Re48, Ro76}, for either
PL, $C^r$ or $C^\infty$ category (this is the more surprising because the connected sum is {\it not well-defined}
for links $S^1\sqcup S^1\to S^3$, cf. (b)).
A non-trivial part of such a proof corresponds to the proof of the Standardization Lemma \ref{l:stand}.b for $X=D^0_+$.
In another formalization the non-trivial part corresponds to proving that `if long knots are isotopic through knots, then they are isotopic through long knots'.

Similarly, it was not so easy for me to reconstruct omitted proof of \cite[Lemma 1.3.b]{Ha66A} which is
Standardization Lemma \ref{l:stand}.b for $X=D^0_+$ and $m\ge q+3$.

Even if this proof is unpublished, it should be known in folklore.

This proof and its generalization from $X=D^0_+$ to $X=D^p_+$ is not hard, see \S\ref{0prosta}.

(b) The Standardization and the Group Structure Lemmas \ref{l:stand} and \ref{t:grst} for $X=S^p$ generalize the well-known construction of the connected sum of links with two {\it numbered oriented} components, i.e. of isotopy classes of embeddings $S^q\sqcup S^q\to S^m$
\cite[Theorem 2.4]{Ha66C}.
It would be nice to have a published example showing that such a connected sum is not well-defined for $m=q+2=3$,
cf. Remark \ref{r:direst}.d for $p=0$ and \cite[Remark before Problem 3.3]{PS96}.
As in (a), a proof that the connected sum is well-defined for $m\ge q+3$ also seems to be unpublished,
cf. \cite[2.5]{Ha66C}; a non-trivial part of such a proof corresponds to the proof of the Standardization Lemma \ref{l:stand}.b for $X=S^0$.

This proof and its generalization from $X=S^0$ to $X=S^p$ is not hard, although more complicated than for $X=D^0_+$, see \S\ref{0prosta}.

(c) 
Also well-known are connected sum group structures on the set $C^{q+2}(S^q)$ of concordance classes of embeddings $S^q\to S^{q+2}$ and on the set $LM_{q,q}^m$ of link homotopy classes of link maps $S^q\sqcup S^q\to S^m$ \cite[Proposition 2.3]{Ko88}.
(See definition of concordance in \S\ref{0prosta}.
In knot theory concordance is called `cobordism', but I use `concordance' to agree with the rest of topology.)
\end{Remark}

\begin{Remark}[the dimension restrictions in the Group Structure Lemma \ref{t:grst}]\label{r:dire}
(a) {\it The orbits of the action $\#$ consist of different number of elements}, so
{\it there are no group structures $(+,0)$ on $E^m(T^{p,q})$ such that $f\#g=f+(0\#g)$}

$\bullet$ for $p+1=q=2k$ and $m=2p+q+2=6k$ \cite[Classification Theorem and Higher-dimensional Classification Theorem]{Sk06'}.

$\bullet$ for $q=3$, $p=1$ and $m=2p+q+2=7$ \cite{CS16}.

$\bullet$ for $q=p=2$ and $m=2p+q+1=7$ \cite[\S1]{CS08}.

(b) {\it There are no group structures on $E^7(T^{1,3})$ such that $r\tau:\pi_3(V_{4,2})\to E^7(T^{1,3})$ is a homomorphism.}
{\it There are no group structures on $E^7_\#(T^{1,3})$ such that $q_{\#}r\tau:\pi_3(V_{4,2})\to E^7_\#(T^{1,3})$ is a homomorphism.}
This is so because $r\tau$-preimages of distinct elements consist of different number of elements, and because of the analogous assertion
for $q_{\#}r\tau$ \cite{CS16}.

(c) {\it There are no group structures on $E^m(T^{p,q})$ such that the Whitney invariant $W$ \cite[\S2]{Sk06}, \cite[\S5]{Sk16e} is a homomorphism}

$\bullet$ for $p+1=q=2k$ and $m=2p+q+2=6k$ because $W$-preimages of distinct elements consist of different number
of elements \cite[Classification Theorem and Higher-dimensional Classification Theorem]{Sk06'}.
In this case $W:E^m_\#(T^{p,q})\to\Z$ is a 1--1 correspondence.

$\bullet$ for $p=q=2k$ and $m=2p+q+1=6k+1$ because by \cite[Theorem 5.1]{Bo71} $\im W=\Z\times0\cup0\times\Z$  which is not a subgroup of the range $\Z^2$ of $W$.
So in this case there are no group structures on $E^m_\#(T^{p,q})$ such that $W:E^m_\#(T^{p,q})\to\Z^2$ is a homomorphism.

(Note that $\varkappa=2W$, where $\varkappa$ is defined in \cite[\S3.9]{Bo71}, \cite[\S2.3]{CS08}.)
\end{Remark}

\begin{Remark}[the dimension restrictions in the Standardization Lemma \ref{l:stand}]\label{r:direst}
(a) The analogue of the Standardization Lemma \ref{l:stand}.a for $X=S^p$, $m=3$ and $p=q=1$ holds and follows from the unknottedness of $S^2$ in $S^3$ and the Alexander Theorem stating that
{\it any embedding $T^{1,1}\to S^3$ extends to an embedding $S^1\times D^2\to S^3$ or $D^2\times S^1\to S^3$.}

(b) If $p\ge q$, $m\ge2p+q+3$ and $X$ is either $D^p_+$ or $S^p$, then $|E^m(X\times S^q)|=1$ (by (D) of
\S\ref{0metca} for $X=D^p_+$, and by the Haefliger-Zeeman Unknotting Theorem \cite[Theorem 2.8.b]{Sk06} for $X=S^p$).
So the Standardization, the Group Structure and the Triviality Lemmas \ref{l:stand}.a, \ref{t:grst} and \ref{l:tricri} hold obviously.

(c) If $X=S^p$, then the analogue of the Standardization Lemma \ref{l:stand}.a

$\bullet$ holds for $m=2p+q+2$ by the proof in \S\ref{0prosta}; so there is a useful multivalued operation on
$E^m(T^{p,q})$, cf. \cite[\S2.4, Definition of parametric connected sum $f+_s\tau(l,b)$]{CS16}.

$\bullet$ is false for $m=2p+q+1$.
(If it were true, then a multivalued operation `$+$' on $E^m(T^{p,q})$ would
be defined by the formula of the Group Structure Lemma \ref{t:grst}.
We would have $W(h)=W(f)+W(g)$ for each $h\in\{f+g\}$, cf. \cite[end of \S2]{Sk08}.
This contradicts to $\im W=\Z\times0\cup0\times\Z$ of Remark \ref{r:dire}.c for $m=2p+q+1$.)

(d) The analogue of the Standardization Lemma \ref{l:stand}.b for $X=S^p$ and $m=2p+q+2$

$\bullet$ holds for embeddings from $\im\cs$;
so there is a group structure on $\im\cs$, cf. \cite[Construction of $\Delta^2$ and $B^6$ in the proof of Lemma 19]{Av14}.

$\bullet$ is false for $p>0$.
(If it were true, then the Group Structure Lemma \ref{t:grst} would be true.
Indeed, for the deduction in \S\ref{0progro} of the latter from the Standardization Lemma \ref{l:stand} a weaker restriction $m\ge\max\{2p+q+2,p+q+3\}$ is sufficient, cf. Remark \ref{r:pl}.a.
So one obtains a contradiction to Remark \ref{r:dire}.ac for $m=2p+q+2$.)

$\bullet$ is conjecturally false for $p=0$.
\end{Remark}

\subsection{Proof of Theorem \ref{t:main} using Lemmas \ref{l:stand} and \ref{t:grst}} \label{0modulo}


Before reading this subsection a reader might want to grasp the idea by reading the proof of a simpler result in
\S\ref{0newpro}
(although the proof here is formally independent of
\S\ref{0newpro}).


\begin{Lemma}\label{l:strip}
For $m\ge p+q+3$ the following is exact sequence of groups:
$$\dots\to E^{m+1}(T^{p,q+1}_+)\overset{\lambda'}\to \pi_q(S^l) \overset{\mu'}\to E^m(T^{p+1,q}_+)
\overset{\nu'}\to E^m(T^{p,q}_+)\to\dots$$
Here $\nu'$ is the restriction-induced map; $\lambda'[f]$ is the obstruction to the existence of a vector field on $f(1_p\times S^{q+1})$ normal to $f(T^{p,q+1}_+)$,
and $\mu'$ is the composition of $\tau$ and the map $\mu'':\pi_q(S^l)=\pi_q(V_{l+1,1})\to\pi_q(V_{m-q,p+1})$ induced by `adding $p$ vectors' inclusion.
\end{Lemma}


{\bf Accurate definition of $\lambda'$.}
We follow \cite[Definition of Ob in p.9]{Sk11}, cf. definition of $\overline\lambda$ in
\S\ref{0prolex}.
Take an embedding $f:T^{p,q+1}_+\to S^{m+1}$.
For sufficiently small $\varepsilon>0$ take

$\bullet$ a trivialization $t:D^p_+\times D^{q+1}_-\times D^{l+1}\to S^{m+1}$ of the normal bundle to
$f(D^p_+\times D^{q+1}_-)$ such that $|t(x,y,z)-f(x,y)|=\varepsilon$ for each
$(x,y,z)\in D^p_+\times D^{q+1}_-\times S^l$;

$\bullet$ a unit normal to $f(T^{p,q+1}_+)$ vector field $s:D^{q+1}_+\to S^{m+1}$ on $f(1_p\times D^{q+1}_+)$.

Let $\lambda'[f]$ be the homotopy class of the map
$$S^q\overset{\theta}\to\partial D^{q+1}_+\overset{t^{-1}s}\to D^p_+\times D^{q+1}_-\times S^l\overset{\pr_3}\to S^l,\quad\text{where}\quad \theta(x):=(0,x).$$
This is well-defined because both trivialization $t$ and vector field $s$ are unique up to homotopy.


\begin{proof}[Proof of Lemma \ref{l:strip}] Consider the following diagram.
$$\xymatrix{
E^{m+1}(T^{p,q+1}_+) \ar@{~>}[r]^{\overline\rho} \ar@{.>}[d]^{\lambda'} & E^{m+1}(S^{q+1}) \ar[d]^\xi \ar@{~>}[rd]^{\overline\xi}\\
\pi_q(S^l) \ar@{-->}[r]^{\mu''} \ar@{.>}[rd]^{\mu'} & \pi_q(V_{m-q,p+1}) \ar@{-->}[r]^{\nu''} \ar[d]^\tau &
\pi_q(V_{m-q,p}) \ar@{-->}[rd]^{\lambda''} \ar@{~>}[d]^{\overline\tau} \\
& E^m(T^{p+1,q}_+) \ar@{.>}[r]^{\nu'} \ar[rd]^\rho & E^m(T^{p,q}_+) \ar@{.>}[r]^{\lambda'} \ar@{~>}[d]^{\overline\rho} &
\pi_{q-1}(S^l) \ar@{-->}[d]^{\mu''}\ar@{.>}[rd]^{\mu'}\\
& &  E^m(S^q) \ar[r]^\xi \ar@{~>}[rd]^{\overline\xi} & \pi_{q-1}(V_{m-q,p+1}) \ar[r]^\tau \ar@{-->}[d]^{\nu''}
& E^{m-1}(T^{p+1,q-1}_+) \ar@{.>}[d]^{\nu'} \\
& & & \pi_{q-1}(V_{m-q,p}) \ar@{~>}[r]^{\overline\tau} & E^{m-1}(T^{p,q-1}_+)
}.$$
Here

$\bullet$ the $\mu''\nu''\lambda''$ sequence is the exact sequence of the `forgetting the last vector' bundle
$S^l\to V_{m-q,p+1}\to V_{m-q,p}$;

$\bullet$ the exact $\tau\rho\xi$- and $\overline\tau\overline\rho\overline\xi$-sequences are defined in Theorem \ref{l:xitaurho}.

Let us prove the commutativity.

Let us prove that $\xi\overline\rho=\mu''\lambda'$ for the left upper square.
By the Standardization Lemma \ref{l:stand}.a each element of $E^{m+1}(T^{p,q+1}_+)$ is representable by a standardized embedding
$f:D^p_+\times S^{q+1}\to S^{m+1}$.
Since $f|_{D^p_+\times D^{q+1}_-}=\inc$, there is a normal $(m-q)$-framing of $f(D^{q+1}_-)$ extending $f|_{D^p_+\times D^{q+1}_-}$ and a normal $(p+1)$-framing of $f(D^{q+1}_+)$ extending $f|_{D^p_+\times D^{q+1}_+}$.
Then $\xi[f|_{0_p\times S^q}]=\mu''\lambda'[f]$ by definitions of $\lambda'$ and $\xi$.

The relation $\lambda''=\overline\tau\lambda'$ follows by definitions of $\lambda'$ and of $\lambda''$ (recalled after the proof).

The commutativity of other squares and triangles is obvious.

Clearly, $\lambda'\nu'=0$.
So the exactness of the $\lambda'\mu'\nu'$ sequence follows by the Snake Lemma, cf. \cite[proof of (6.5)]{Ha66A}.
\end{proof}

{\bf Definition of $\lambda''$.}
A map $\R^k\times X\to \R^n$ is called a {\it linear monomorphism} if its restriction to $\R^k\times x$ is a linear monomorphism for each $x\in X$.
Represent an element $x\in\pi_q(V_{m-q,p})$ by a linear monomorphism
$$f:\R^p\times S^q\to\R^{m-q}=\R^p\times\R^{l+1}\quad\text{such that}\quad f(x,y)=(x,0)\quad\text{for each}\quad y\in D^q_-.$$
By the Covering Homotopy Property for the `forgetting last vector' bundle $V_{m-q,p+1}\to V_{m-q,p}$ the restriction $f|_{\R^p\times D^q_+}$ extends to a linear monomorphism $s:\R^p\times\R\times D^q_+\to\R^{m-q}=\R^p\times\R^{l+1}$.
For each $y\in S^{q-1}$ we have $s(\R^p\times0\times y)=\R^p\times0$, so $s(0_p,1,y)\not\in\R^p\times0$.
Hence we can define a map
$$g:S^{q-1}\to S^l\quad\text{by}\quad g(y):=
n\pr\phantom{}_2s(0_p,1,y),\quad\text{where}\quad n(z):=z/|z|.$$
Let $\lambda''(x)$ be the homotopy class of $g$.

\smallskip
{\bf Definition of the Zeeman homomorphism}
$$\zeta=\zeta_{m,n,q}:\pi_q(S^{m-n-1})\to E^m_U(S^q\sqcup S^n)\quad\text{for}\quad q\le n.$$
For a map $x:S^q\to S^{m-n-1}$ define an embedding
$$\overline\zeta\phantom{}_x:1\times S^q\cup(-1)\times S^n\to S^m\quad\text{by}\quad
\overline\zeta\phantom{}_x(t,a) :=
\begin{cases} \inc_{m,m-n-1,n}(0,a) & t=-1\\ \inc_{m,m-n-1,n}(x(a),a) & t=1 \end{cases}.$$
Recall that $S^q\sqcup S^n = 1\times S^q\sqcup(-1)\times S^n$.
Define $\zeta[x]:=[\overline\zeta_x]$.


Clearly, $\zeta$ is well-defined, is a homomorphism, and $\lambda\zeta=\id\pi_q(S^{m-n-1})$,
see \cite[Theorem 10.1]{Ha66C}, \cite[Definition of Ze before Lemma 3.4]{Sk11}, \cite[Remarks 3.2.ac]{Sk16h}.

Note that $\zeta_{m,q,q}=r\tau^m_{0,q}=r\mu'$.

\smallskip
{\bf Definition of the homomorphism}
$$\sigma=\sigma_{m,p,q}:E^m_0(S^q\sqcup S^{p+q})\to E^m(T^{p,q})\quad\text{for}\quad m\ge p+q+3
\quad\text{and}\quad q>0.$$
Let $E^m_0(S^q\sqcup S^n)\subset E^m(S^q\sqcup S^n)$ be the subset formed by the isotopy classes of embeddings
whose restriction to the {\it first} component is standard.
Represent an element of $E^m_0(S^q\sqcup S^{p+q})$ by an embedding
$$f:S^q\sqcup S^{p+q}\to S^m\quad\text{such that}\quad f|_{S^q}=\inc|_{0_{p+1}\times S^q}
\quad\text{and}\quad f(S^{p+q})\cap \inc(D^{p+1}\times S^q)=\emptyset.$$
Join $f(S^{p+q})$ to $\inc(-1_p\times S^q)$ by an arc whose interior misses $f(S^{p+q})\cup\inc(D^{p+1}\times S^q)$.
Let $\sigma[f]$ be the isotopy class of the embedded connected sum of $\inc|_{S^p\times S^q}$ and
$f|_{S^{p+q}}$ along the arc.
(The images of these embeddings are not necessarily contained in disjoint balls.)
For $p=0$ the orientation on $\inc(-1_p\times S^q)$ is `parallel' to the orientation on $\inc(1_p\times S^q)$.
Cf. \cite[\S3, Definition of $\sigma^*$]{Sk11}.

\smallskip
The map $\sigma$ is well-defined for $m\ge p+q+3$ and is a homomorphism for $m\ge 2p+q+3$
\cite[Lemmas 3.1--3.3]{Sk11}.
(The main reason for being well-defined is that $E^m_0(S^q\sqcup S^n)$ is in 1--1 correspondence with the set
of isotopy classes of embeddings whose restriction to the first component is {\it standard}
\cite[Lemma 3.1]{Sk11}, cf. \cite[Proof of Theorem 7.1]{Ha66C}.)

We have $\sigma(f)+\cs g=\sigma(f\#g)$ \cite[Remark after Lemma 3.3]{Sk11}.

For $p=0$
some results on $\sigma$ are presented in \cite{Sk24}.


\begin{proof}[Proof of Theorem \ref{t:main}]
Clearly, the first sentence follows from the `moreover' part.
So let us prove the `moreover' part.
Consider the following diagram.
$$\xymatrix{
& \pi_q(S^l) \ar[r]^{\mu'} \ar[d]^{\zeta}
& E^m(T^{p+1,q}_+) \ar[dr]^{\nu'} \ar@{..>}[d]_{q_{\#}r} & & \pi_{q-1}(S^l) \ar[d]^{\zeta}\\
E^{m+1}(T^{p,q+1}_+)
\ar[ur]^{\lambda'} \ar[r]_{\zeta\lambda'} \ar[dr]_{i\zeta\lambda'}
& E^m_U(S^q\sqcup S^{p+q}) \ar[r]_{\sigma_\#} \ar@{-->}@(ru,rd)[u]_{\lambda} \ar[d]^i
& E^m_\#(T^{p,q}) \ar[r]_(.6){\overline\nu}
& E^m(T^{p,q}_+) \ar[r]_(.4){\zeta\lambda'} \ar[ru]^{\lambda'}  \ar[dr]_{i\zeta\lambda'}
& E^{m-1}_U(S^{q-1}\sqcup S^{p+q-1}) \ar@{-->}@(ru,rd)[u]_{\lambda} \ar[d]^i\\
& E^m_0(S^q\sqcup S^{p+q}) \ar[r]^\sigma & E^m(T^{p,q}) \ar[u]^{q_{\#}} \ar[ru]^\nu & & E^{m-1}_0(S^{q-1}\sqcup S^{p+q-1}) \\
&  E^m(S^{p+q}) \ar[u]_{\phantom{}_j\#} \ar[ur]_{\cs}
}$$
Here the $\lambda'\mu'\nu'$-sequence is defined in Lemma \ref{l:strip}, maps $\zeta$ and $\sigma$ are defined above,

$\bullet$ $i$ is the inclusion,

$\bullet$ $\nu$ is the restriction-induced map,

$\bullet$ $\sigma_\#:=q_\#\sigma$,

$\bullet$ the map $\overline\nu$ is well-defined by $\overline\nu q_{\#}(f):=\nu(f)$,

$\bullet$ $\phantom{}_j\#g:=j\# g$, where the `standard embedding' $j:S^q\sqcup S^{p+q}\to S^m$ is any embedding
whose components are contained in disjoint balls and are isotopic to the inclusions.

The commutativity of the triangles is clear, except for $\cs =\sigma\phantom{}_j\#$, which follows by $\sigma[j]=[\inc]$.

The map $\overline\mu$ of Theorem \ref{t:main} is well-defined by $\overline\mu(\mu'x):=\sigma_\#\zeta x$.
Let $X:=\ker\nu'=\im\mu'$.

Recall the Serre Theorem: {\it the group $\pi_q(S^l)$ is finite unless $q=4k-1$ and $l=2k$ for some $k$, and $\pi_{4k-1}(S^{2k})$ is the sum of $\Z$ and a finite group.}
This and (VF) of \S\ref{0metca} imply the assertion on the finiteness of $X=\im(\tau\mu'')$.
Then using the exact sequence of the `forgetting the last vector' bundle
$S^l\to V_{m-q,p+1}\to V_{m-q,p}$ we obtain the assertion on $X^{6k+p}_{p,4k-1}$.

It suffices to prove that the horizontal sequence of Theorem \ref{t:main} is exact.

The exactness of the $\nu\sigma(i\zeta\lambda')$-sequence is \cite[Theorem 1.6]{Sk11}.

The map $i\oplus\phantom{}_j\#$ is an isomorphism \cite[Theorem 2.4]{Ha66C}.
Hence by the Smoothing Lemma \ref{t:smo} and $\cs =\sigma\phantom{}_j\#$, the first two third-line groups both have $E^m(S^{p+q})$-summands mapped one to the other under $\sigma$.
Taking quotients by these summands one obtains the exactness of the second line.

Since $\lambda\zeta=\id$, we have that $\zeta$ is injective and
$E^m_U(S^q\sqcup S^{p+q})=\im\zeta\oplus K^m_{q,p+q}$ (the mutually inverse isomorphisms are given by $x\mapsto(\zeta\lambda x,x-\zeta\lambda x)$ and $(y,z)\mapsto y+z$).

Since the right $\zeta$ is injective, we have $\im\nu'=\ker\lambda'=\ker\zeta\lambda'=\im\overline\nu$.

The restriction $\sigma_\#|_{K^m_{q,p+q}}$ is injective because
$$\ker\sigma_\#\cap K^m_{q,p+q}=\im(\zeta\lambda')\cap K^m_{q,p+q}\subset\im\zeta\cap K^m_{q,p+q}=0.$$
If $\overline\mu\mu'x=\sigma_\#\zeta x=0$, then $x\in\im\lambda'=\ker\mu'$, so $\mu'x=0$.
Hence $\overline\mu$ is injective.

Also
$$\ker\overline\nu=\im\sigma_\#=\sigma_\#\im\zeta\oplus \sigma_\#K^m_{q,p+q}=\im \overline\mu\oplus \sigma_\#K^m_{q,p+q}.$$
Thus the horizontal sequence of Theorem \ref{t:main} is exact.
\end{proof}

{\bf Comment.}
I conjecture that the piecewise linear (PL) analogue of Theorem \ref{t:main} holds.
This analogue is obtained by replacing $E^m_\#(T^{p,q})$ and $E^m(T^{p+1,q}_+)$ by $E^m_{PL}(T^{p,q})$  and $E^m_{PL}(T^{p+1,q}_+)$; the group $K^m_{q,p+q}$ remains the same.

The PL analogue of the Standardization Lemma \ref{l:stand} for $X=S^p$ holds by \cite{Sk05i}.
The PL analogues of the Standardization Lemma \ref{l:stand} for $X=D^p_+$, of the Group Structure and the Triviality Lemmas \ref{t:grst} and \ref{l:tricri} hold with the same proof (there is even a simplification in the proof that $[f]+[\overline f]=0$).
It would be interesting to find the PL analogue of Theorem \ref{l:xitaurho}.

\subsection{Proof of the Standardization Lemma \ref{l:stand}}\label{0prosta}


{\it Proof of (a) for $X=D^p_+$.}
Take an embedding $g:T^{p,q}_+\to S^m$.
Since every two embeddings of a disk into $S^m$ are isotopic, we can make an isotopy of $S^m$ and assume that
$g=\inc$ on $D^p_+\times D^q_-$.

The ball $D^m_-$ is contained in a tubular neighborhood of $\inc(D^p_+\times D^q_-)$ in $S^m$ relative to
\linebreak
$\inc(D^p_+\times\partial D^q_-)$.
The image $g(D^p_+\times\Int D^q_+)$ is disjoint from some tighter such tubular neighborhood.
Hence by the Uniqueness of Tubular Neighborhood Theorem we can make an isotopy of $S^m$ and assume that $g(D^p_+\times\Int D^q_+)\cap D^m_-=\emptyset$.
Then $g$ is standardized.
\qed

\smallskip
{\it Proof of (b) for $X=D^p_+$.}
Take an isotopy $g$ between standardized embeddings $T^{p,q}_+\to S^m$.
By the 1-parametric version of `every two embeddings of a disk into $S^m$ are isotopic' we can make a self-isotopy of
$\id S^m$, i.e. a level-preserving autodiffeomorphism of $S^m\times I$ identical on $S^m\times\{0,1\}$, and assume that $g=\inc\times\id I$ on $D^p_+\times D^q_-\times I$.

The ball $D^m_-\times I$ is contained in a tubular neighborhood $V$ of $\inc(D^p_+\times D^q_-)\times I$ in
$S^m\times I$ relative to $\inc(D^p_+\times\partial D^q_-)\times I$.
We may assume that $V\cap S^m\times k$ is `almost $D^m_-$' for each $k=0,1$.

The image $g(D^p_+\times\Int D^q_+\times I)$ is disjoint from some tighter such tubular neighborhood, whose intersection with $S^m\times k$ is $V\cap S^m\times k$ for each $k=0,1$.
Hence by the Uniqueness of Tubular Neighborhood Theorem we can make an isotopy of $S^m\times I$ relative to
$S^m\times\{0,1\}$ and assume that $g(D^p_+\times\Int D^q_+\times I)\cap D^m_-\times I=\emptyset$.
Then $g$ is standardized.
\qed

\begin{figure}[h]
\definecolor{uuuuuu}{rgb}{0.27,0.27,0.27}
\begin{tikzpicture}[line cap=round,line join=round,>=triangle 45,x=1.0cm,y=1.0cm]
\clip(-5.18,-0.5) rectangle (5.41,3.26);

\draw [rotate around={-0.5:(0.43,0.76)}] (0.43,0.76) ellipse (2.49cm and 0.99cm);
\fill [white] (2.5,1.1) rectangle (2.7,1.4);
\fill [white] (-1.7,1.1) rectangle (-2,1.4);

\draw [rotate around={-0.83:(0.36,1.51)}] (0.36,1.51) ellipse (2.33cm and 1.09cm);

\draw (0.3,0.42)-- (0.3,-0.23);
\draw (0.3,0.26)-- (0.53,0.26);
\draw (0.3,0.09)-- (0.53,0.09);
\draw (0.3,-0.06)-- (0.53,-0.06);
\draw (0.47,0.3)-- (0.53,0.26);
\draw (0.53,0.26)-- (0.47,0.22);
\draw (0.53,0.09)-- (0.47,0.13);
\draw (0.53,0.09)-- (0.47,0.05);
\draw (0.53,-0.06)-- (0.47,-0.02);
\draw (0.53,-0.06)-- (0.47,-0.11);
\begin{scriptsize}
\fill [color=uuuuuu] (0.3,-0.23) circle (1.0pt);
\fill [color=uuuuuu] (0.3,0.42) circle (1.0pt);
\draw (-0.1,0.1) node {\large $\Delta_1$};
\draw (2.9,2.6) node {\large $g(T^{p,q})$};
\end{scriptsize}
\end{tikzpicture}
\caption{To the proof of the Standardization Lemma \ref{l:stand}.a for $X=S^p$}
\label{f:stand}
\end{figure}

Extend $\inc$ to $\sqrt2 D^{p+1}\times D^{q+1}$ by the same formula as in the definition of $\inc$.
For $\gamma\le\sqrt2$ denote $\Delta_\gamma:=\inc(\gamma D^{p+1}\times\{-1_q\})\subset \Int D^m_-$.

\smallskip
{\it Proof of (a) for $X=S^p$.} See fig. \ref{f:stand}.
Take an embedding $g:T^{p,q}\to S^m$.
Since $m>2p+q$, every two embeddings $S^p\times D^q\to S^m$ are isotopic (this is a trivial case of Theorem \ref{l:xitaurho}).
So we can make an isotopy and assume that $g=\inc$ on $S^p\times D^q_-$.

Since $m>2p+q+1$, by general position we may assume that $\im g\cap \Delta_1=\partial \Delta_1$.
Then there is $\gamma$ slightly greater than 1 such that $\im g\cap \Delta_\gamma=\partial \Delta_1$.
Take the `standard' $q$-framing on $\Delta_\gamma$ tangent to $\inc(\gamma D^{p+1}\times S^q)$
whose restriction to $\partial\Delta_1$ is the `standard' normal $q$-framing of $\partial\Delta_1$ in $\im g$.
Then the `standard' $(m-p-q-1)$-framing normal to $\inc(\gamma D^{p+1}\times S^q)$ is an
$(m-p-q-1)$-framing on $\partial\Delta_1$ normal to $\im g$.
Using these framings we construct

$\bullet$ an orientation-preserving embedding $H:D^m_-\to D^m_-$ onto a tight neighborhood of $\Delta_1$ in $D^m_-$,
and

$\bullet$ an isotopy $h_t$ of $\id T^{p,q}$ shrinking $S^p\times D^q_-$ to a tight neighborhood of
$S^p\times\{-1_q\}$ in $S^p\times D^q_-$ such that
$$H(\Delta_{\sqrt2})=\Delta_\gamma,\quad H\inc(S^p\times D^q_-)=H(D^m_-)\cap\im g
\quad\text{and}\quad H\inc=\inc h_1 \quad\text{on}\quad S^p\times D^q_-.$$
Embedding $H$ is isotopic to $\id D^m_-$ by \cite[Theorem 3.2]{Hi76}.
This isotopy extends to an isotopy $H_t$ of $\id S^m$ by the Isotopy Extension Theorem \cite[Theorem 1.3]{Hi76}.
Then $H_t^{-1}gh_t$ is an isotopy of $g$.
Let us prove that embedding $H_1^{-1}gh_1$ is standardized.

We have $H_1^{-1}gh_1=H_1^{-1}\inc h_1=\inc$ on $S^p\times D^q_-$.
Also if $H_1^{-1}gh_1(S^p\times \Int D^q_+)\not\subset \Int D^m_+$, then there is $x\in S^p\times \Int D^q_+$
such that $gh_1(x)\in H(D^m_-)$.
Then $gh_1(x)=H\inc(y)=\inc h_1(y)=gh_1(y)$ for some $y\in S^p\times D^q_-$.
This contradicts to the fact that $gh_1$ is an embedding.
\qed

\smallskip
An embedding $F:N\times I\to S^m\times I$ is a {\it concordance} if $N\times k=F^{-1}(S^m\times k)$ for each $k=0,1$.
Embeddings are called {\it concordant} if there is a concordance between them.

\smallskip
{\it Proof of (b) for $X=S^p$.}
Take an isotopy $g$ between standardized embeddings.
The restriction $g|_{S^p\times D^q_-}$ is an isotopy between standard embeddings.
So this restriction gives an embedding $g':S^p\times D^q_-\times S^1\to S^m\times S^1$ homotopic to
$\inc|_{S^p\times0}\times \id S^1$.
Since $m+1>2(p+1)$, by general position $g'|_{S^p\times0\times S^1}$ is isotopic to $\inc|_{S^p\times0}\times\id S^1$.
Since $m>2p+q+1$, the Stiefel manifold $V_{m-p,q}$ is $(p+1)$-connected.
Hence every two maps $S^1\times S^p\to V_{m-p,q}$ are homotopic.
Therefore $g'$ is isotopic to $\inc\times \id S^1$.
So  we can make a self-isotopy of $\id S^m$, i.e. a level-preserving autodiffeomorphism of $S^m\times I$ identical on $S^m\times\{0,1\}$, and assume that $g=\inc\times\id I$ on $S^p\times D^q_-\times I$.


Since $m>2p+q+2$, by general position we may assume that $\im g\cap\Delta_1\times I=\partial\Delta_1\times I$.
Then there is a disk $\Delta\subset D^m_-\times I$ such that
$$\Int\Delta\supset\Delta_1\times(0,1),\quad \Delta\cap D^m_-\times\{0,1\}=\Delta_{\sqrt2}\times\{0,1\}
\quad\text{and}\quad \im g\cap \Delta=\partial\Delta_1\times I.$$
Take the `standard' $q$-framing on $\Delta$ tangent to $\inc(\sqrt2D^{p+1}\times S^q)\times I$
whose restriction to $\partial\Delta_1\times I$ is the `standard' normal $q$-framing of $\partial\Delta_1\times I$ in $\im g$.
Then the `standard' $(m-p-q-1)$-framing on $\partial\Delta_1\times I$ normal to
$\inc(\sqrt2D^{p+1}\times S^q)\times I$ is an $(m-p-q-1)$-framing on $\partial\Delta_1\times I$ normal to $\im g$.
Using these framings we construct

$\bullet$ an orientation-preserving embedding $H:D^m_-\times I\to D^m_-\times I$ onto a neighborhood of
$\Delta_1\times I$ in $D^m_-\times I$,
and

$\bullet$ an isotopy $h_t$ of $\id T^{p,q}\times I$ shrinking $S^p\times D^q_-\times I$ to a neighborhood of
$S^p\times\{-1_q\}\times I$ in $S^p\times D^q_-\times I$ such that
$$H(\Delta_{\sqrt2}\times I)=\Delta,\quad H(\inc(S^p\times D^q_-)\times I)=H(D^m_-\times I)\cap\im g$$
$$\text{and}\quad H\circ(\inc\times\id I)=(\inc\times\id I)\circ h_1 \quad\text{on}\quad S^p\times D^q_-\times I.$$
Analogously to the proof of (a) embedding $H$ is isotopic to $\id(D^m_-\times I)$,
such an isotopy extends to an isotopy $H_t$ of $\id(S^m\times I)$, and $H_t^{-1}gh_t$ is an isotopy from $g$
to a standardized isotopy $H_1^{-1}gh_1$.
\qed

\subsection{Proof of the Group Structure Lemma \ref{t:grst}}\label{0progro}

Let us prove that {\it the sum is well-defined}, i.e. that for standardized embeddings $f,g:X\times S^q\to S^m$ the isotopy class of $s_{f,g}$ depends only on $[f]$ and $[g]$, but not on the chosen standardizations of $f_0$ and $f_1$.
For this let us define parametric connected sum of isotopies.
Take isotopic standardized embeddings $f,f'$ and $g,g'$.
By the Standardization Lemma \ref{l:stand}.b there are standardized isotopies
$F,G:X\times S^q\times I\to S^m\times I$ between $f$ and $f'$, $g$ and $g'$.
Define an isotopy
$$S:X\times S^q\times I\to S^m\times I\quad\text{by}
\quad S(x,y,t)=\begin{cases} F(x,y,t)& y\in D^q_+\\
R(G(x,Ry,t))      &y\in D^q_-\end{cases}.$$
Then $S$ is an isotopy between $s_{fg}$ and $s_{f'g'}$.

(The isotopy class of $S$ may depend on $F,G$ not only on their isotopy classes.)

Clearly, {\it $\inc:X\times S^q\to S^m$ represents  the zero element.}

Denote by $R^t$ be the rotation of $\R^s=\R^2\times\R^{s-2}$ whose restriction to the plane $\R^2\times0$ is the rotation through the angle $+\pi t$ and which leaves the orthogonal complement fixed.

Let us prove {\it the commutativity}.
Each embedding $f:T^{p,q}\to\R^m$ is isotopic to $R^{-t}\circ f\circ (\id S^p\times R^t)$.
Hence the embedding $s_{fg}$ is isotopic to $R^1\circ s_{fg} \circ(\id S^p\times R^1)=s_{gf}$.

Let us prove {\it the associativity}.
Define $D^m_{++}\subset S^m$ by equations $x_1\ge0$ and $x_2\ge0$.
Define $D^m_{+-},D^m_{-+},D^m_{--}$ analogously.
Analogously to (or by) the Standardization Lemma \ref{l:stand}.a each element in $E^m(X\times S^q)$ has a representative $f$ such that
$$f(S^p\times D^q_{++})\subset D^m_{++}\quad\text{and}\quad f|_{S^p\times(D^q-D^q_{++})}=\inc.$$
Let $f,g,h:X\times S^q\to S^m$ be such representatives of three elements of $E^m(X\times S^q)$.
Then both $[f]+([g]+[h])$ and $([f]+[g])+[h]$ have a representative defined by
$$s_{fgh}(x,y):=
\begin{cases} f(x,y)& y\in D^q_{++} \\
R_{1/2}(g(x,R_{1/2}y))& y\in D^q_{+-} \\
R_1(h(x,R_1y))& y\in D^q_{--} \\
\inc(x,y)& y\in D^q_{-+}
\end{cases}.$$

Let us prove that $[\overline f]+[f]=0$.
Clearly, $\overline f$ is standardized.
Embedding $s_{f\overline f}$ can be extended to an embedding $X\times D^{q+1}\to D^{m+1}$ as follows.
(Analogous minor modification should be done in \cite[1.6]{Ha66A} because the embedding $D^{n+1}\to D^{n+q+1}$ constructed there is not orthogonal to the boundary.)
Represent an element of  $D^u$ as $(a_0,a)$, where $a_0\in [-1,1]$ and $a\in\sqrt{1-a_0^2}D^{u-1}$.
Define
$$\gamma:D^u\to D^u\quad\text{by}\quad \gamma(a_0,a):=\left(\frac{a_0(1+|a|^2)}{1+\sqrt{1-a_0^2-a_0^2|a|^2}},a\right) \quad\text{and}$$
$$H:X\times D^{q+1}\to D^{m+1}\quad\text{by}\quad
H(x,\gamma(y)):=\gamma(f(x,y)).$$
Since $RR_1=R_2$, the map $H$ is well-defined.
Using $\pr_1\gamma(a_0,a)=\frac1{a_0}-\sqrt{\frac{1-a_0^2}{a_0^2}-|a|^2}$ for $a_0\ne0$,
one can check that $H$ is a smooth embedding (in particular, orthogonal to the boundary).
Hence embedding $s_{f\overline f}$ is isotopic to $\inc$ by the following Triviality Lemma \ref{l:tricri}.
\qed

\begin{Lemma}[Triviality Lemma]\label{l:tricri}
Let $X$ denote either $D^p_+$ or $S^p$.
Assume that $m\ge 2p+q+3$ for $X=S^p$ and  $m\ge q+3$ for $X=D^p_+$.
An embedding $X\times S^q\to S^m$ is isotopic to $\inc$ if and only if it extends to an embedding
$X\times D^{q+1}\to D^{m+1}$.
\end{Lemma}

{\it Proof.}
For each $u$ observe that $D^u\cong \Delta^u:=D^u\bigcup\limits_{S^{u-1}=S^{u-1}\times0} S^{u-1}\times I$.

The `only if' part is proved by `capping' an isotopy $X\times S^q\times I\to S^m\times I$ to $\inc$, i.e.  by taking its union with $\inc:X\times D^{q+1}\to D^{m+1}$.

(The `only if' part does not require the dimension assumption.)

Let us prove the `if' part.
Represent the extension as an embedding $f:X\times\Delta^{q+1}\to\Delta^{m+1}$ such that
$f^{-1}(\partial\Delta^{m+1})=X\times\partial\Delta^{q+1}$.
Analogously to the Standardization Lemma \ref{l:stand}.a
$f$ is isotopic relative to $X\times\partial\Delta^{q+1}$ to an embedding $g$ such that
$$g(X\times S^q\times I)\subset S^m\times I\quad\text{and}
\quad g=\inc\quad\text{on}\quad X\times D^{q+1}.$$
Then the restriction $g:X\times S^q\times I\to S^m\times I$ of $g$ is a concordance from given embedding
$g_0:X\times S^q\to S^m$ to $\inc$.

If $X=S^p$, then $m\ge p+q+3$.
Hence concordant embeddings $g_0$ and $\inc$ are isotopic by \cite[Corollary 1.4]{Hu70} for $Q=S^m$, $X_0=Y=\emptyset$.

If $X=D^p_+$, then $m\ge q+3$.
Hence the concordance $g:1_p\times S^q\times I\to S^m\times I$
is isotopic to an isotopy  relative to $1_p\times S^q\times\{0,1\}$ by \cite[Theorem 1.5]{Hu70} for $Q=S^m$, $X_0=Y=\emptyset$.
The latter isotopy has a normal $p$-framing that coincides with the $p$-framing on $1_p\times S^q\times\{0,1\}$
defined by  $g_0|_{1_p\times S^q}$ and $\inc|_{1_p\times S^q}$.
The vectors at $(x,t)\in S^m\times I$ are tangent to $S^m\times\{t\}$.
So the latter isotopy extends to an isotopy between $g_0$ and $\inc$.
\qed

\begin{Remark}\label{r:pl}
(a) The dimension restriction in the Triviality Lemma \ref{l:tricri} for $X=S^p$ could be relaxed to $m\ge\max\{2p+q+2,p+q+3\}$.
This is so by Remark \ref{r:direst}.c and because we use analogue of the Standardization Lemma \ref{l:stand}.a not \ref{l:stand}.b.

(b) The analogue of the Triviality Lemma \ref{l:tricri} $m=q+2$ is false (for $X=D^2_+, S^1,D^1_+,S^0$).
Indeed, take an embedding $f:S^q\to S^{q+2}$ non-isotopic to the standard embedding but concordant to it. 
The latter means that $f$ extends to a proper embedding $D^{q+1}\to D^{q+3}$. 
The latter has trivial normal bundle, which defines an embedding $D^2_+\times D^{q+1}\to D^{q+3}$ whose restriction to $D^2_+\times S^q$ an embedding to $S^{q+2}$.
The restrictions of the latter embedding to $X\times S^q$, where $X=D^2_+, S^1,D^1_+,S^0$, are not isotopic to $\inc$, but extend to embeddings $X\times D^{q+1}\to D^{q+3}$. 
(Details of this argument appeared in a discussion with T. Garaev.)

(c) The analogue of the Triviality Lemma \ref{l:tricri} for $X=S^p$, $p\ge2$ and $m=2p+q+1$ is false.
Indeed, let $x\in\pi_p(V_{p+q+1,q+1})\cong\Z_{(p)}$ be a generator.
The map $\tau$ is defined before Corollary \ref{t:corlam}.
Then the embedding $\tau^m_{q,p}(x)|_{T^{q,p}}:T^{q,p}\to S^m$ is not isotopic to $\inc$ but extends to an embedding $T^{q+1,p}_+\to D^{m+1}_+$.
Indeed,

$\bullet$ the required extension is $\tau^{m+1}_{q+1,p}(\widehat x)|_{T^{q+1,p}_+} : T^{q+1,p}_+\to D^{m+1}_+$, where $\widehat x$ is the image  of $x$ under the stabilization-induced map $\pi_p(V_{p+q+1,q+1})\to\pi_p(V_{p+q+2,q+2})$;

$\bullet$ the non-existence of isotopy follows by \cite[Torus Lemma 6.1]{Sk02} because $x\ne0$ and $2p+q+1\ge \frac{3p}2+q+2$.
\end{Remark}

\subsection{Proof of the Smoothing Lemma \ref{t:smo}}\label{0prosmho}

\begin{Lemma}[proved below]\label{l:smohom}
For $m\ge 2p+q+3$ there is a homomorphism
$$\overline\sigma: E^m(T^{p,q})\to E^m(S^{p+q})\quad\text{such that}\quad \overline\sigma\circ\cs=\id E^m(S^{p+q}).$$
\end{Lemma}

{\it The Smoothing Lemma \ref{t:smo} follows} because Lemma \ref{l:smohom} and $q_{\#}\circ\cs=0$ imply that
$q_{\#}\oplus\overline\sigma$ is an isomorphism.
Lemma \ref{l:smohom} is known \cite[Proposition 5.6]{CRS} except for the non-trivial assertion that $\overline\sigma$ is a homomorphism.

\smallskip
{\it Proof of Lemma \ref{l:smohom}: definition of $\overline\sigma$ and proof that
$\sigma\circ\cs=\id E^m(S^{p+q})$.}
The map $\overline\sigma$ is `embedded surgery of $S^p\times1_q$'.
We give an alternative detailed construction following
\cite[Proposition 5.6]{CRS}.
Take $f\in E^m(T^{p,q})$.
By the Standardization Lemma \ref{l:stand}.a there is a standardized representative $f':T^{p,q}\to S^m$ of $f$.
Identify
$$S^{p+q}\quad\text{and}\quad S^p\times D^q_+\bigcup\limits_{S^p\times \partial D^q_+} D^{p+1}\times \partial D^q_+$$
by a diffeomorphism.
Define an embedding
$$
i':D^{p+1}\times \partial D^q_+\to D^m_-\quad\text{by}\quad i'(x,(0,y)):=(-\sqrt{1-|x|^2},y,0_l,x)/\sqrt2.
$$
Then $i'$ is an extension of the restriction $S^p\times \partial D^q_+\to\partial D^m_+$ of $\inc_{m,p,q}$.
Infinite derivative for $|x|=1$ means that $i'$ meets the boundary regularly.
Hence $i'$ and $f'|_{S^p\times D^q_+}$ form together a ($C^1$-smooth) embedding $g:S^{p+q}\to S^m$.
Let $\overline\sigma(f):=[g]$.

The map $\overline\sigma$ is well-defined for $m\ge 2p+q+3$ by the Standardization Lemma \ref{l:stand}.b because the above construction of $\overline\sigma$ has an analogue for isotopy, cf. beginning of \S\ref{0progro}.

Clearly, $\overline\sigma\circ\cs(g)=\overline\sigma(0\#g)=\overline\sigma(0)+g=0+g=g$.

\begin{figure}[h]
\centerline{\includegraphics[width=7cm]{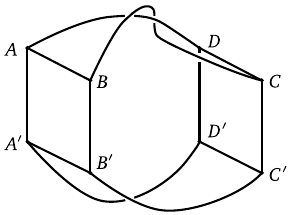}}
\caption{{\it To the proof that $\overline\sigma$ is a homomorphism.}
This picture illustrates the proof by the case $p=0$, $q=1$ and $m=3$
(these values are not within the dimension range $m\ge2p+q+3$).
The part above plane $ABCD$ stands for $\widehat{D^m_+}$.
The part below plane $A'B'C'D'$ stands for $\widehat{D^m_-}$.
The part between the planes stands for $S^{m-1}\times D^1$.
The upper curved lines stand for $f_+(S^p\times S^{q-1})=u(S^p\times S^{q-1})$.
The bottom curved lines stand for $f_-(S^p\times S^{q-1})=u(S^p\times S^{q-1})$.
The union of segments $A'A,B'B,C'C$ and $D'D$ stands for $u(S^p\times S^{q-1}\times D^1)$.
The union of segments $A'A$ and $B'B$ stands for $\inc(S^p\times 1_{q-1})\times D^1$.
The quadrilateral $A'ABB'$ stands for the `surgery disk' $\inc(D^{p+1}\times D^{q-1}_+)\times D^1$.
The union of the upper curved lines and the segment $AB$ stands for the $(p+q)$-disk $\Delta_+$.
Analogously for $\Delta_-$.
The union of $\Delta_+,\Delta_-$ and the segments $C'C$ and $D'D$ stands for the $(p+q)$-sphere
that is the image of a representative of $\overline\sigma[u]$.
The union of $\Delta_+$ and $CD$ stands for $\Sigma_+$.
Analogously for $\Sigma_-$.
The quadrilateral $C'CDD'$ stands for the tube $\inc(D^{p+1}\times D^{q-1}_-)\times D^1$.
}
\label{f:homom}
\end{figure}

\smallskip
{\it Proof of Lemma \ref{l:smohom}: beginning of the proof that $\overline\sigma$ is a homomorphism.}
See fig. \ref{f:homom}.
Denote by $\widehat A$ a copy of $A$.
For each $n$ identify
$$S^n\quad\text{with}\quad
\widehat{D^n_+} \bigcup\limits_{\widehat{\partial D^n_+}=S^{n-1}\times1} S^{n-1}\times D^1
\bigcup\limits_{S^{n-1}\times\{-1\}=\widehat{\partial D^n_-}} \widehat{D^n_-}.$$
Then $S^{n-1}=S^{n-1}\times0\subset S^n$.
Let $\inc=\inc_{m-1,p,q-1}$.
Under the identifications $\widehat{\partial D^n_\pm}=S^{n-1}\times\{\pm1\}$, $n\in\{m,q\}$, the embedding $\inc_{m,p,q}$ goes to $\inc|_{S^p\times S^{q-1}}$.
Hence analogously to (or by) the Standardization Lemma \ref{l:stand}.a each element in $E^m(T^{p,q})$ has a representative
$f$ such that

$\bullet$ $f(S^p\times \widehat{D^q_+})\subset\widehat{D^m_+}$;

$\bullet$ $f=\inc_{m,p,q}$ on $S^p\times \widehat{D^q_-}$ (the image of this embedding lyes in $\widehat{D^m_-}$);

$\bullet$ $f=\inc|_{S^p\times S^{q-1}}\times\id D^1$ on $S^p\times S^{q-1}\times D^1$ (the image of this embedding lies in $S^{m-1}\times D^1$).


Take embeddings $f_\pm:T^{p,q}\to S^m$ satisfying the above properties.
Then $[f_+]+[f_-]$ has a representative $u:T^{p,q}\to S^m$ such that

$\bullet$ $u=f_+$ on $S^p\times\widehat{D^q_+}$;

$\bullet$ $u=(\id S^p\times R)\circ f_-\circ(\id S^p\times R)$ on $S^p\times\widehat{D^q_-}$;

$\bullet$ $u=\inc|_{S^p\times S^{q-1}}\times\id D^1$ on $S^p\times S^{q-1}\times D^1$.


\smallskip
For completion of the proof that $\overline\sigma$ is a homomorphism we need an equivalent definition of $\overline\sigma$.

First we assume that $p=0$, i.e. define the embedded connected sum of embeddings $f_{-1},f_1:S^q\to S^m$ whose images are disjoint.
Take an embedding $l:D^1\times D^q_-\to S^m$ such that
$$l=f_k\quad\text{on}\quad k\times D^q_-\quad\text{and}\quad l(D^1\times D^q_-)\cap f_k(S^q)=l(k\times D^q_-)\quad\text{for}\quad k=\pm1.$$
Define $h:S^q\to S^m$ by
$$
h(x):=\begin{cases}f_0(x) &x\in \widehat{D^q_+} \\ l(x) &x\in D^1\times\partial D^q_+ \\ f_1(x) &x\in \widehat{D^q_-} \end{cases}.
$$
Then a representative of $[f_0]+[f_1]$ is obtained from $h$ by smoothing of the `dihedral corner' along
$h(S^0\times\partial D^q_+)$.
This smoothing is local replacement of embedded $(I\times0\cup0\times I)\times D^{q-1}$ by
embedded $C\times D^{q-1}$, where $C\subset I^2$ is a smooth curve joining $(0,1)$ to $(1,0)$ and such that
$C\cup[1,2]\times0\cup0\times[1,2]$ is smooth.
This smoothing is `canonical', i.e. does not depend on the choice of $C$.
Cf. \cite[Proof of 3.3]{Ha62k} and, for non-embedded version, \cite{U}.

Let us generalize this definition to arbitrary $p$.
Given embedding $f:T^{p,q}\to S^m$, take an embedding $l:D^{p+1}\times D^q_-\to S^m$ such that
$$l=f\quad\text{on}\quad S^p\times D^q_-\quad\text{and}\quad l(D^{p+1}\times D^q_-)\cap f(T^{p,q})=l(S^p\times D^q_-).$$
Define $h:S^{p+q}\to S^m$ by
$$
h(x):=\begin{cases}f(x) &x\in S^p\times D^q_+ \\ l(x) &x\in D^{p+1}\times\partial D^q_+ \end{cases}.
$$
Then a representative of $\overline\sigma(f)$ is obtained from $h$ by `canonical' smoothing of the `dihedral corner' along
$h(S^p\times \partial D^q_+)$ analogous to the above case $p=0$.

This definition is equivalent to that from the beginning of proof of Lemma \ref{l:smohom}
because there are a closed $\varepsilon$-neighborhood $U$ of the image of $l$ (for some small $\varepsilon>0$)
and a self-diffeomorphism $G:S^m\to S^m$ such that
$G(D^m_-,\inc(T^{p,q}_-),i'(D^{p+1}\times \partial D^q_+))=(U,U\cap f(T^{p,q}),U\cap h(S^{p+q}))$.

The result of the above surgery does not depend on the choices involved because $\overline\sigma(f)$ is well-defined.

\smallskip
{\it Completion of the proof that $\overline\sigma$ is a homomorphism.}
(This argument appeared after a discussion with A. Zhubr, cf. \cite{Zh}.)
Recall that a representative of $\overline\sigma[u]$ is obtained from $u$ by `embedded surgery of $\inc(S^p\times1_{q-1})\times0$'.
Recall that the isotopy class of an embedding $g:S^{p+q}\to S^m$ is defined by the image of $g$ and an orientation on the image.

Denote
$$\Delta_\pm\ :=\ u(S^p\times D^q_\pm)\ \cup\ \inc(D^{p+1}\times D^{q-1}_+)\times\{\pm1\}
\ \cong\ S^p\times D^q\cup D^{p+1}\times D^{q-1}_+\ \underset{PL}\cong\ D^{p+q}.$$
Then the oriented image of the representative of $\overline\sigma[u]$ is obtained by `canonical' smoothing of corners from
$$\left(u(T^{p,q})-\inc(S^p\times D^{q-1}_+)\times D^1\right)
\cup(\inc\times\id D^1)\left(D^{p+1}\times\partial(D^{q-1}_+\times D^1)\right)=$$
$$=\Delta_-\cup \inc\partial(D^{p+1}\times D^{q-1}_-)\times D^1\cup\Delta_+
\ \underset{PL}\cong\ D^{p+q}\times0\cup S^{p+q-1}\times I\cup D^{p+q}\times 1\ \underset{PL}\cong\ S^{p+q}.$$

This oriented $(p+q)$-sphere is a connected sum of oriented $(p+q)$-spheres
$$\Sigma_\pm:=\Delta_\pm\cup \inc(D^{p+1}\times D^{q-1}_-)\times\{\pm1\}
\ \underset{PL}\cong\ 0\times D^{p+q}\cup D^{p+q}_+\ \underset{PL}\cong\ S^{p+q}$$
along the tube $\inc(D^{p+1}\times D^{q-1}_-)\times D^1$.
The image of  a representative of $\overline\sigma[f_\pm]$ is obtained from $\Sigma_\pm$ by `canonical' smoothing of the `dihedral corner'.
The corners of the tube $\inc(D^{p+1}\times D^{q-1}_-)\times D^1$ can be `canonically' smoothed to obtain an embedding $D^{p+q}\times D^1\to S^m$.
Thus $\overline\sigma[u]=\overline\sigma[f_+]+\overline\sigma[f_-]$.
\qed

\subsection{Appendix: new direct proof of Corollary \ref{t:corlam}.a} \label{0newpro}


{\bf Definition of the following diagram for $m\ge p+q+3$.}
$$\xymatrix{
\pi_{q+1}(V_{m-q,p}) \ar[d]^{\tau_p} \ar[r]^{\lambda''} & \pi_q(S^l) \ar[r]^{\mu''}  \ar[rd]_{\sigma'} &
\pi_q(V_{m-q,p+1}) \ar[r]^{\nu''} \ar[d]^{q_{\#}r\tau_{p+1}} & \pi_q(V_{m-q,p}) \ar[r]^{\lambda''} \ar[d]^{\tau_p} & \pi_{q-1}(S^l) \\
E^{m+1}(T^{p,q+1}_+) \ar[ru]_{\lambda'} &  & E^m_\#(T^{p,q}) \ar[r]_{\overline\nu} & E^m(T^{p,q}_+) \ar[ru]_{\lambda'} }.$$
Here $q_{\#},r,\lambda'$ and $\tau_p:=\tau^m_{p,q}$ are defined in \S\S \ref{0stat}, \ref{0meth}, \ref{0modulo},

$\bullet$ the $\mu''\nu''\lambda''$ sequence is the exact sequence of the `forgetting the last vector' bundle
$S^l\to V_{m-q,p+1}\to V_{m-q,p}$,

$\bullet$ the map $\overline\nu$ is well-defined by $\overline\nu q_{\#}\pr[f]:=[f|_{T^{p,q}_+}]$.

\smallskip
{\bf Definition of $\sigma'=\sigma'_p$.}
For a map $x:S^q\to S^l$ representing an element of $\pi_q(S^l)$ let
$\zeta_x$ be the composition
$$D^{p+1}\times S^q\overset{\inc_{p+q,p,q}\times x}\to D^{p+q+1}\times S^l\overset{\inc}\to S^m\overset{h}\to S^m,$$
where $\inc:=\inc\phantom{}_{m,p+q,l}$ and $h$ is the exchange of the first $l+1$ and the last $p+q+1$ coordinates.
The image of $\zeta_x$ is contained in $h\inc(D^{p+q+1}\times S^l)$ and so is disjoint from
$S^{p+q}=h\inc(S^{p+q}\times0_{l+1})$.
Join $S^{p+q}$ to $\zeta_x(-1_p\times S^q)$ by an arc whose interior misses $S^{p+q}\cup\zeta_x(D^{p+1}\times S^q)$.
Let $\sigma'[x]$ be the equivalence class of
the embedded connected sum of the inclusion $S^{p+q}\subset S^m$ and $\zeta_x|_{T^{p,q}}$ along the arc.
Clearly, $\sigma'$ is well-defined for $m\ge p+q+3$, and is a homomorphism
(and $\sigma'=\sigma_\#(h^{-1}\zeta)$ in the notation of \S\ref{0modulo}).

\begin{Lemma}\label{l:exact}
(a) The $\sigma'\overline\nu\lambda'$-sequence is exact for $2m\ge p+3q+4$.

(b) The right $\tau_p$ is an isomorphism for $2m\ge3q+5$ and an epimorphism $2m\ge3q+4$.

(c) The diagram is commutative up to sign for $m\ge2p+q+2$ and $2m\ge 2p+3q+4$.
\end{Lemma}

Corollary \ref{t:corlam}.a for $p\ge1$ follows from Lemma \ref{l:exact} and 5-lemma.
Corollary \ref{t:corlam}.a for $p=0$ follows from Lemma \ref{l:exact} because $V_{m-q,0}$ is a point, so $q_{\#}r\tau_1$ is the composition of the isomorphisms $\sigma'$ and $(\mu'')^{-1}$.
(This proof is simpler than the proof sketched in \S\ref{0meth} for $p=1$, but is more complicated than the classical proof for $p=0$.)

Lemma \ref{l:exact}.a follows from (L) of \S\ref{0metca} and \cite[Theorem 1.6]{Sk11}
(restated in the proof of Theorem \ref{t:main} in \S\ref{0modulo}).
Cf. \S\ref{0prolex}.

Lemma \ref{l:exact}.b follows by (D) of \S\ref{0metca}, either applying (trivial case of) Theorem \ref{l:xitaurho} or deducing from $E^{m+1}(S^{q+1})=0$ that every isotopy $S^q\times I\to S^m\times I$ between standard embeddings is isotopic to the identical one.

An embedding $f:S^n\times X\to S^m$ is {\it reflection-symmetric} if $f\circ(R_1\times\id X)=R_1\circ f.$

\smallskip
{\it Proof of Lemma \ref{l:exact}.c.}
Clearly, $\tau_p\nu''=\overline\nu q_{\#}r\tau_{p+1}$.
By definitions of $\lambda'$  and $\lambda''$ (\S\ref{0modulo}) $\lambda''=\tau_p\lambda'$.
So it remains to prove that $\sigma'_p=q_{\#}r\tau_{p+1}\mu''_p$.
Take any $x\in\pi_q(S^l)$.

First we prove the case $p=0$ for $2m\ge3q+4$.
Take embeddings $f':T^{0,q}\to S^m$ representing $\sigma'_0x$,
and $f:T^{0,q}\to S^m$ representing $r\tau_1\mu''_0$.
Denote by $\lambda_-f$ the linking coefficient, i.e. the homotopy class of the {\it second} component
in the complement to the first component.
Define $\lambda_-f'$ analogously.
Clearly, $\lambda_-f'=\lambda_-f$.
The restrictions of $f$ and $f'$ to the first component are isotopic to the standard embedding.
Since $2m\ge3q+4$, by (L) of \S\ref{0metca} $q_{\#}[f']=q_{\#}[f]$.

A representative of $r\tau_{p+1}\mu''_px$ is a reflection-symmetric extension of a representative of
$r\tau_p\mu''_{p-1}x$.
A representative of $\sigma'_px$ is a reflection-symmetric extension of a representative of $\sigma'_{p-1}x$.
Since $m\ge2p+q+2$ and $\sigma'_0=q_{\#}r\tau_1\mu''_0$ for $2m\ge3q+4$, we have $\sigma'_p=q_{\#}r\tau_{p+1}\mu''_p$ for $2m\ge 2p+3q+4$ by the following Lemma \ref{l:ext}.
\qed



\begin{Lemma}\label{l:ext}
If $m\ge2p+q+2$ and $f_p,g_p:T^{p,q}\to S^m$ are reflection-symmetric extensions of concordant embeddings
$f,g:T^{p-1,q}\to S^{m-1}$, then $f_p$ and $g_p$ are concordant.
\end{Lemma}


{\it Proof.}
Let $F$ be a concordance between $f$ and $g$.
Denote by $\widehat A$ a copy of $A$.
For each $n$ identify (as in the proof of Lemma \ref{l:smohom})
$$S^n\quad\text{with}\quad
\widehat{D^n_+}\bigcup\limits_{\widehat{\partial D^n_+}=S^{n-1}\times1}
S^{n-1}\times D^1
\bigcup\limits_{S^{n-1}\times\{-1\}=\widehat{\partial D^n_-}}\widehat{D^n_-} \quad\text{and}
$$
$$T^{p,q}\quad\text{with}\quad \widehat{T^{p,q}_+} \bigcup\limits_{\widehat{T^{p-1,q}}=T^{p-1,q}\times0} T^{p-1,q}\times I
\bigcup\limits_{T^{p-1,q}\times1=\widehat{T^{p-1,q}}} \widehat{T^{p,q}_-}.$$
Let $F':T^{p,q}\to S^m$ be an embedding obtained from $f_p|_{T^{p,q}_+}\cup F\cup g_p|_{T^{p,q}_+}$ by these identifications.
Define an embedding
$$i_\varepsilon:D^q\to S^{q+1}\quad\text{by}\quad i_\varepsilon(x):=(\varepsilon x,\sqrt{1-\varepsilon^2|x|^2}).$$
Since $m\ge 2p+q+2$, analogously to the Standardization Lemma \ref{l:stand}.a there are $\varepsilon>0$ and an embedding
$$F'':D^{p+1}_+\times D^q\to S^m\quad\text{such that}\quad
F''|_{S^p\times D^q}=F'\circ(\id S^p\times i_\varepsilon).$$
Let
$$\Sigma := (T^{p,q}-S^p\times i_\varepsilon(\Int D^q))
\bigcup\limits_{S^p\times i_\varepsilon(S^{q-1})=\widehat{T^{p,q-1}}} \widehat{T^{p+1,q-1}_+}.$$
Clearly, $\Sigma\cong S^{p+q}$.
Identify
$$T^{p+1,q}_+\quad\text{with}\quad D^{p+1}_+\times D^q
\bigcup\limits_{D^{p+1}\times S^{q-1}=\widehat{T^{p+1,q-1}_+}\times0} \con\Sigma.$$
Take a piecewise smooth  embedding $T^{p+1,q}_+\to S^m\times I$ obtained by this identification from
the `union' of $F''$ and the cone $\con\Sigma\to S^m\times I$ over $(F'\cup\psi)|_\Sigma$.
This embedding can be shifted to a proper piecewise smooth concordance
$F_+$ between $f_p|_{T^{p,q}_+}$ and $g_p|_{T^{p,q}_+}$, smooth outside a ball.

The complete obstruction to smoothing $F_+$ is in $E^m(S^{p+q})$ \cite{BH70, Bo71}.
If we change concordance $F$ by connected sum with an embedding $h:S^{p+q}\to S^{m-1}\times I$,
then $F_+$ changes by a connected sum with the cone over $h$.
Hence the obstruction to smoothing $F_+$ changes by adding $[h]\in E^m(S^{p+q})$ \cite{BH70, Bo71}.
Therefore by changing $F$ modulo the ends we can make $F_+$ a smooth concordance (in particular, orthogonal to the boundary).
So we may assume that $F_+$ is a smooth concordance.

Define $F_-$ by symmetry to $F_+$.
The two proper concordances $F_+$ and $F_-$ fit together to give the required concordance between $f_p$ and $g_p$.
\qed

\subsection{Appendix: new direct proof of Lemma \ref{l:exact}.a} \label{0prolex}

Here we present a simpler direct proof of Lemma \ref{l:exact}.a for $m\ge2p+q+3$ and $2m\ge 2p+3q+4$ (this weaker result is sufficient for Corollary \ref{t:corlam}.a).
Such a proof is recovered from \cite[Restriction Lemma 5.2]{Sk06''} and illustrates ideas of proof of the deeper result \cite[Theorem 1.6]{Sk11} whose full strength is used in \S\ref{0modulo}.

Take the `standard' homotopy equivalence $h:S^m-\inc(1_p\times S^q)\to S^{m-q-1}$.

\smallskip
{\it Definition of $\lambda(f)\in\pi_{p+q}(S^{m-q-1})$ for an embedding $f:T^{p,q}\to S^m$ coinciding with $\inc$
on $T^{p,q}_+$.}
The restrictions of $f$ and $\inc$ onto $T^{p,q}_-$ coincide on the boundary and so form
a map
$$\widehat f:T^{p,q}\to S^m-\inc(1_p\times S^q)\overset h\to S^{m-q-1}.$$
Clearly, $\widehat f|_{\inc(-1_p\times S^q)}$ is null-homotopic.
Since $m\ge 2p+q+3$, by general position the map
$$T^{p,q}\to \frac{T^{p,q}}{-1_p\times S^q\vee S^p\times 1_q}\cong S^{p+q}$$
induces a 1--1 correspondence between $\pi_{p+q}(S^{m-q-1})$ and homotopy classes of maps $T^{p,q}\to S^{m-q-1}$ null-homotopic on $-1_p\times S^q$.
Let $\lambda(f)$ be the homotopy class corresponding to the homotopy class of $-1_p\times S^q$.
\footnote{In this definition of $\lambda(f)$ one can avoid using the triviality of the restriction to
$\inc(-1_p\times S^q)$, analogously to the definition of $\overline\lambda(f)$ given below.
However, the above definition is more convenient for Lemma \ref{l:isot}.
\newline
{\it Alternative definition of $\lambda(f)$.}
Denote $B^{p+q}:=T^{p,q}-(\Int D^p_+\times S^q\cup S^p\times \Int D^q_+).$
Since $m\ge 2p+q+2$, by general position making an isotopy we may assume that $f=\inc$ on $T^{p,q}-\Int B^{p+q}$.
Hence $f|_{B^{p+q}}$ and $\inc|_{B^{p+q}}$ form
a map $S^{p+q}\to S^m-\inc(1_p\times S^q)\overset h\to S^{m-q-1}$.
Let $\lambda(f)$ be the homotopy class of this map.
Since $m\ge 2p+q+3$, this is well-defined, i.e., is independent of the isotopy used in the definition.
\newline
There is analogous {\it alternative definition of $\overline\lambda(f)$}.
\newline
The element $\lambda(f)$ is not an isotopy invariant of $f$, as opposed to $\overline\lambda(f)$.}

\smallskip
{\it Definition of $\overline\lambda(f)\in\pi_{p+q-1}(S^{m-q-1})$ for an embedding $f:T^{p,q}_+\to S^m$
coinciding with $\inc$ on $1_p\times S^q$.}
Since $m\ge 2p+q+2$, by general position the map
$$\partial T^{p,q}_+\overset{\cong}\to T^{p,q-1}\to T^{p,q-1}/S^p\times 1_q\overset{\cong}\to S^{p+q-1}\vee S^{q-1}$$
induces a 1--1 correspondence between homotopy classes of maps $\partial T^{p,q}_+\to S^{m-q-1}$ and
$S^{p+q-1}\vee S^{q-1}\to S^{m-q-1}$.
Let $\overline\lambda(f)$ be the image under the restriction map of the homotopy class corresponding to the homotopy class of the map
$$f|_{\partial T^{p,q}_+}:\partial T^{p,q}_+\to S^m-\inc(1_p\times S^q)\overset h\to S^{m-q-1}.$$

\begin{Lemma}\label{l:isot} Assume that $m\ge2p+q+3$ and $2m\ge 2p+3q+4$.

(a) If embeddings $f,g:T^{p,q}\to S^m$ coincide with $\inc$ on $T^{p,q}_+$ and $\lambda(f)=\lambda(g)$, then $q_{\#}[f]=q_{\#}[g]$.

(b) For each $y\in\pi_q(S^l)$ there is an embedding $g:T^{p,q}\to S^m$ representing $\sigma'(y)$,
coinciding with $\inc$ on $T^{p,q}_+$ and such that $\lambda(g)=\pm\Sigma^py$.

(c) For each embedding $f:T^{p,q}_+\to S^m$ coinciding with $\inc$ on $1_p\times S^q$ we have $\overline\lambda(f)=\pm\Sigma^p\lambda'(f)$.
(Recall that $\lambda'$ is defined in \S\ref{0modulo}.)
\end{Lemma}

{\it Proof of (a).}
Since $\lambda(f)=\lambda(g)$, the restrictions $f,g:T^{p,q}_-\to S^m-\inc(1_p\times S^q)$ are homotopic relative to the boundary.
Since $q-1\ge 2(p+q)-m+1$ and $m-q-2\ge 2(p+q)-m+2$, these restrictions
are PL isotopic by \cite{Ir65}.
Since $2m\ge3q+4$, there is a unique obstruction to smoothing such a PL isotopy, assuming values in $E^m(S^{p+q})$.
This obstruction can be killed by embedded connected sum of $f$ with an embedding $S^{p+q}\to S^m$.
Hence $q_{\#}[f]=q_{\#}[g]$.
\qed

\smallskip
{\it Proof of (b).} Take a map $y:S^q\to S^l$ representing $y$.
Take the linear homotopy between the map $S^l\to 0_{l+1}\in\R^{l+1}$ and the composition of $y$ with the inclusion $S^l\subset\R^{l+1}$.
This homotopy defines an isotopy between $\inc$ and the embedding $\zeta_y$ from the definition of $\sigma'$.
This isotopy is `covered' by an isotopy of $S^m$.
The latter isotopy carries the representative of $\sigma'(y)$ constructed in the definition to an embedding $T^{p,q}\to S^m$ which we denote by $g$.
Clearly, $g=\inc$ on $T^{p,q}_+$.
Clearly, $\lambda(g)$ equals to the homotopy class of `the image' of $S^{p+q}$ under this isotopy in $S^m-\inc(1_p\times S^q)\sim S^{m-q-1}$.
This equals to the homotopy class of $S^{p+q}$ in $S^m-\zeta_y(1_p\times S^q)\sim S^{m-q-1}$.
Since $q\le2l-2$, by the Freudenthal Suspension Theorem the group $\pi_q(S^l)$ is stable.
Hence by \cite[Lemma 5.1]{Ke59} the latter homotopy class equals to the $\pm\Sigma^p$-image of
the homotopy class of $\zeta_y(1_p\times S^q)$ in $S^m-S^{p+q}\sim S^l$.
The latter class equals to $y$, so $\lambda(g)=\pm\Sigma^py$.
\qed

\smallskip
{\it Proof of (c).}
(This is a simpler `algebraic' analogue of \cite[Lemma 3.8]{Sk11}.)

Take an embedding $f:T^{p,q}_+\to S^m$.
Since $2m\ge3q+4$, making an isotopy of $S^m$ we may assume that $f=\inc$ on $1_p\times S^q$.
Clearly, $\overline\lambda(f)$ is the homotopy class of the composition
$$S^{p+q-1}\overset{\cong}\to \partial D^p_+\times D^q_-\cup D^p_-\times\partial D^q_-\overset{f\cup\inc}\to S^m-\inc(1_p\times S^q)\overset h\to S^{m-q-1},$$
where the map $f\cup\inc$ is formed by the restrictions
of  $f$ and $\inc$ which agree on the boundary.

Identify in the natural way

$\bullet$ $S^{p+q-1}$ with $\partial D^p_+\times D^q_-\cup D^p_-\times\partial D^q_-$;

$\bullet$ $D^p_+\times D^q_-$ with $D^q_-\times D^p_+$; and

$\bullet$ $S^{m-q-1}\quad\text{with}\quad\partial(D^p_+\times D^{l+1})=\partial D^p_+*S^l,
\quad\text{where}\quad X*Y=\dfrac{X\times I\times Y}{X\times 0\times y,x\times 1\times Y}$.

Use the notation $t,s$ from definition of $\lambda'$ in \S\ref{0modulo}.
Since $f=\inc$ on $1_p\times S^q$, we may assume that $ht=\pr_2$ on $D^q_+\times\partial(D^p_+\times D^{l+1})$.
So the above composition representing $\overline\lambda(f)$ maps

$\bullet$ $\partial D^p_+\times D^q_-=f^{-1}t(D^q_-\times\partial D^p_+\times0_{l+1})$ to
$\partial D^p_+\times 0\times[S^l]\subset\partial D^p_+*S^l$ as $\pr_2t^{-1}f$;

$\bullet$ $-1_p\times\partial D^q_-$ to $[\partial D^p_+]\times1\times S^l\subset\partial D^p_+*S^l$ as $\pr_3t^{-1}s\pr_2$;

$\bullet$ $D^p_-\times\partial D^q_-$ `linearly' to the join $\partial D^p_+*S^l$.

Therefore $\pm\overline\lambda(f)=\Sigma^p[\pr_3t^{-1}s\theta]=\Sigma^p\lambda'(f)$.
\qed

\smallskip
In the rest of this subsection we prove Lemma \ref{l:exact}.a for $m\ge2p+q+3$ and $2m\ge 2p+3q+4$

\smallskip
{\it The exactness at $E^m_\#(T^{p,q})$.}
Clearly, $\overline\nu\sigma'=0$.

Let $f:T^{p,q}\to S^m$ be an embedding such that $\overline\nu q_{\#}[f]=0$.
Making an isotopy of $S^m$ we may assume that $f=\inc$ on $T^{p,q}_+$.
Since $q\le2l-1$, by the Freudenthal Suspension Theorem there is $x\in\pi_q(S^l)$ such that $\Sigma^px=\lambda(f)$.
So by Lemma \ref{l:isot}.b there is an embedding $g:T^{p,q}\to S^m$ representing $\pm\sigma'(x)$, coinciding with $\inc$ on $T^{p,q}_+$ and such that $\lambda(g)=\Sigma^px$.
Then by Lemma \ref{l:isot}.a $q_{\#}[f]=q_{\#}[g]=\sigma'(\pm x)$.
\qed

\smallskip
{\it The exactness at $E^m(T^{p,q}_+)$.}
Take any $z\in\ker\lambda'$.
Since $2m\ge3q+4$, by Lemma \ref{l:exact}.b the right $\tau_p$ of the diagram
from the beginning of \S\ref{0newpro} is an epimorphism.
Hence there is $z_1\in \pi_q(V_{m-q,p})$ such that $z=\tau_pz_1$.
By Lemma \ref{l:exact}.c $\lambda''z_1=\lambda'z=0$.
Then by exactness there is  $z_2\in \pi_q(V_{m-q,p+1})$ such that $z_1=\nu''z_2$.
Then by Lemma \ref{l:exact}.c $z=\tau_pz_1=\tau_p\nu''z_2=\overline\nu(q_{\#}r\tau_{p+1}z_2)$.

Take an embedding $f:T^{p,q}\to S^m$.
Then $f|_{T^{p,q}_-}$ gives a null-homotopy of $\overline\lambda(f)$.
Since $q\le2l-2$, by the Freudenthal Suspension Theorem the homomorphism $\Sigma^p$ of Lemma \ref{l:isot}.c is injective.
So $\lambda'[f]=0$.
\qed

\begin{figure}[h]
\centerline{\includegraphics[width=8cm]{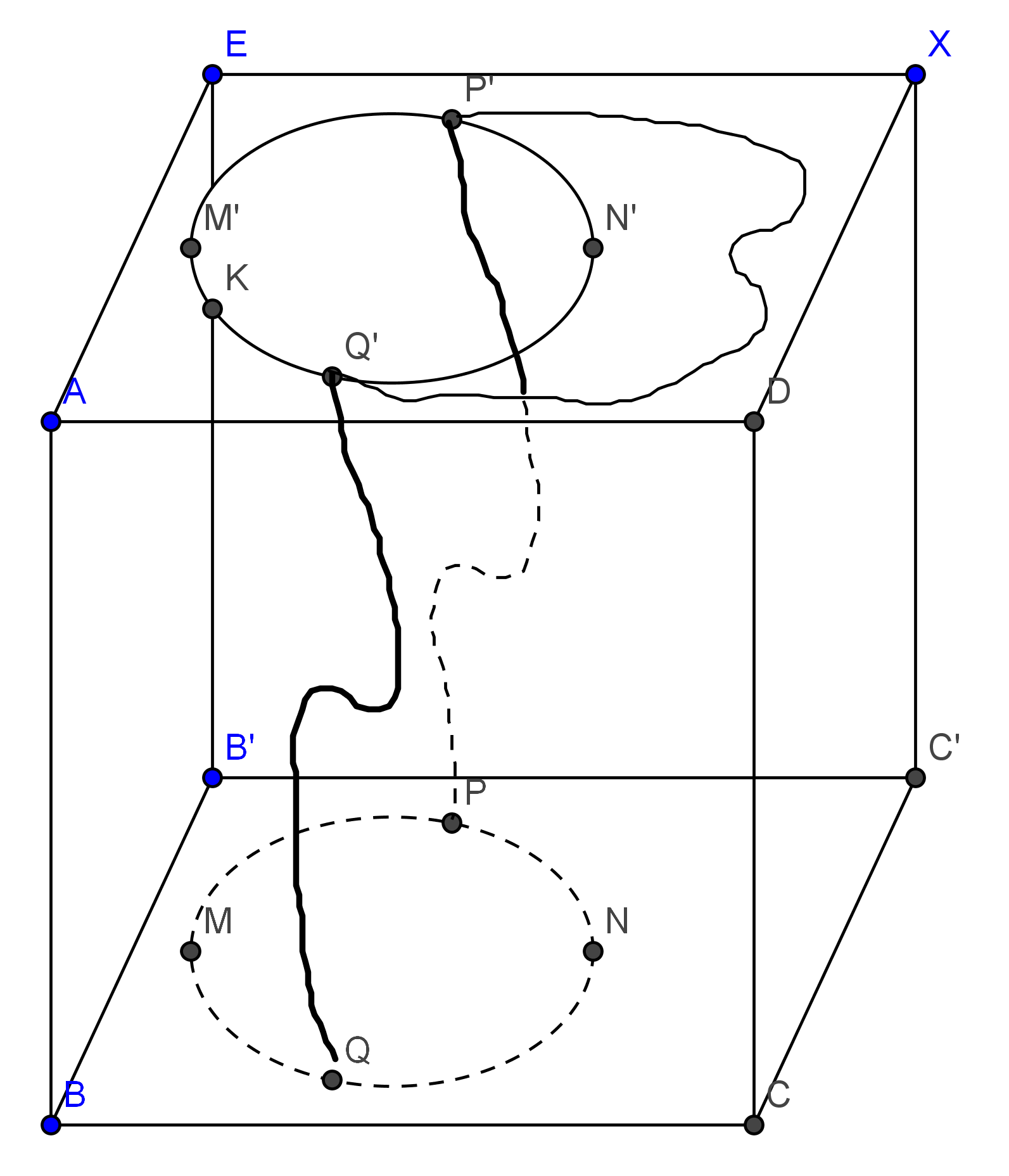}}
\caption{{\it To the proofs that $\ker\sigma'\subset\im\lambda'$ and $\sigma'\lambda'=0$.}
The cube stands for $S^m\times I$; its horizontal faces stand for $S^m\times k$, $k=0,1$.
The ellipses stand for $\inc(T^{p,q})\times k$, $k=0,1$.
The curved line $P'Q'$ stands for $g(D^p_-\times S^q)\times1$.
The curved lines $PP'$ and $QQ'$ stand for the image of $\partial D^p_-\times S^q\times I$
under the isotopy between standard embeddings.
The points $M,N,M',N'$ stand for $\inc(\pm1_p\times S^q)\times k$, $k=0,1$.}
\label{f:lambda}
\end{figure}

\smallskip
{\it Proof that $\ker\sigma'\subset\im\lambda'$.}
Take any $x\in\ker\sigma'$.

If $p=0$, then the linking number of $\sigma'(x)=0$ is $x$.
So $x=0\in\im\lambda'$.

If $p>0$, then $2m\ge2p+3q+4>3q+5$.
Take a representative $g$ of $\sigma'(x)=0$ given by Lemma \ref{l:isot}.b.
The restriction $T^{p,q}_+\times I\to S^m\times I$ of an isotopy between $g$ and the standard embedding is an isotopy between standard embeddings.
So this restriction can be `capped' to give an embedding $G:T^{p,q+1}_+\to S^{m+1}$.
Since $2m\ge3q+5$, we may assume that $G=\inc$ on $1_p\times S^q$.
We have
$$\Sigma^p\lambda'(G)\overset{(1)}=\pm\overline\lambda(G)\overset{(2)}=\pm\lambda(g)\overset{(3)}=\Sigma^px. $$
Here equalities (1) and (3) follow by Lemma \ref{l:isot}.bc  and equality (2) is proved below.
Since $q\le2l-2$, by the Freudenthal Suspension Theorem $\Sigma^p$ injective.
So $x=\pm \lambda'(G)$.

Let us prove equality (2).
Denote $\Delta^{q+1}:=\inc(1_p\times D^{q+1})$.
Then $\Delta^{q+1}\cap\inc(T^{p,q}_-)=\emptyset$ and we may assume that $g$ is transverse to $\Delta^{q+1}$.
Since $q\le2l-2$, by the Freudenthal Suspension Theorem the group $\pi_q(S^l)$ is stable.
So $\lambda(g)$ corresponds under Pontryagin construction and (de)suspension to the framed intersection
$\Delta^{q+1}\cap g(T^{p,q}_-)\subset \Delta^{q+1}$, for certain framings on $\Delta^{q+1}$ and on $g(T^{p,q}_-)$.
Analogously, we may assume that $G$ is transverse to $\Delta^{q+1}\times I\subset S^m\times I\subset S^{m+1}$.
So $\overline\lambda(G)$ corresponds under Pontryagin construction and (de)suspension to the framed intersection
$\Delta^{q+1}\times I\cap G(\partial T^{p,q}_-\times I)\subset \Delta^{q+1}\times I$, for certain framings on $\Delta^{q+1}\times I$ and on $G(\partial T^{p,q}_-\times I)$.
Consider the framed intersection $\Delta^{q+1}\times I\cap G(T^{p,q}_-\times I)\subset \Delta^{q+1}\times I$, for certain framings on $\Delta^{q+1}\times I$ and on $G(T^{p,q}_-\times I)$ of which the above framings are the restrictions.
The latter  framed intersection is a framed cobordism between the first two.
Thus $\overline\lambda(G)=\lambda(g)$.
\qed

\smallskip
{\it Proof that $\sigma'\lambda'=0$.}
\footnote{If $2m\ge3q+5$, by Lemma \ref{l:exact}.b the left $\tau_p$ is an epimorphism.
Then for $2m\ge2p+3q+4$ by Lemma \ref{l:exact}.c $\sigma'\lambda'=0$.}
(This proof suggests an alternative definition of $\zeta\lambda'$ allowing minor simplification in \cite{Sk11}.)
Take an embedding $F:T^{p,q+1}_+\to S^{m+1}$.
Analogously to the Standardization Lemma \ref{l:stand} making an isotopy we can assume that

$\bullet$ $F|_{D^p_+\times S^q\times I}$ is a concordance between standard embeddings,

$\bullet$ $F|_{D^p_+\times D^{q+1}_k}$ is the standard embedding into $D^{m+1}_k$ for each $k=0,1$.

Since $q\le2l-1$, by the Freudenthal Suspension Theorem there is $x\in\pi_q(S^l)$ such that $\Sigma^px=\overline\lambda(F)$.
By Lemma \ref{l:isot}.b there is an embedding $g:T^{p,q}\to S^m$ representing $\pm\sigma'(x)$, coinciding with $\inc$ on $T^{p,q}_+$ and such that $\lambda(g)=\Sigma^px$.

The concordance $F|_{D^p_+\times S^q\times I}$ is ambient, so there is a diffeomorphism
$$\Phi:S^m\times I\to S^m\times I\quad\text{such that}\quad \Phi(F(y,0),t)=
F(y,t)\quad\text{for each }y\in D^p_+\times S^q,\ t\in I.$$
So the standard embedding $\inc$ is concordant to an embedding $f:T^{p,q}\to S^m\times1=S^m$ defined by $f(y):=\Phi(F_0(y),1)$.
The class $\lambda(f)$ corresponds to the homotopy class of
$$\inc|_{T^{p,q}_-}\cup f|_{T^{p,q}_-}:T^{p,q}\to S^m\times 1-\inc(\Int T^{p,q}_+)\times1\sim S^{m-q-1}.$$
We have $\lambda(f)=\overline\lambda(F)=\Sigma^px=\lambda(g)$, where the first equality is analogous to the equality (2) in the previous proof.
Besides, $f$ and $g$ are standard on $T^{p,q}_+$.
Therefore by Lemma \ref{l:isot}.a $\sigma'\lambda'(F)=q_{\#}[g]=q_{\#}[f]=0$.
\qed


\end{document}